\documentclass[reqno,12pt]{amsart}

\usepackage{amsmath,amsfonts,amsthm,amssymb,amsxtra,hyperref,color,graphicx,setspace,xspace,hyperref}

\usepackage[latin1]{inputenc}
\newcommand{\be}{\begin{eqnarray}}
\newcommand{\ee}{\end{eqnarray}}
\newcommand{\ben}{\begin{eqnarray*}}
\newcommand{\een}{\end{eqnarray*}}
\newcommand{\dis}{\displaystyle}
\newcommand{\beq}{\begin{equation}}

\newcommand{\eeq}{\end{equation}}
\newcommand{\R}{{\mathbb R}}
\newcommand{\N}{{\mathbb N}}

\newcommand{\vs}{{\vskip .2cm}}

\newcommand{\sep}{; }
\renewcommand{\(}{\left(}
\renewcommand{\)}{\right)}

\newtheorem{thm}{Theorem}[section]
\newtheorem{coro}[thm]{Corollary}
\newtheorem{prop}[thm]{Proposition}
\newtheorem{lemma}[thm]{Lemma}
\newtheorem{remark}{Remark}

\newtheorem{example}{\emph{Example}}

\renewcommand{\subjclass}[1]{\thanks{2010 \emph{Mathematics Subject Classification.} #1}}

\title[Nodal solutions and weighted $p$-Laplace operator]{Existence of sign changing solutions for an equation with a weighted $p$-Laplace operator}

\author[C.~Cort\'azar, J.~Dolbeault, M.~Garc\'ia-Huidobro and R.~Man\'asevich]{}


\thanks{${}^1$ Supported by Fondecyt 1110074.
${}^2$ Partially supported by Mathamsud 13MATH-03 and ECOS C11E07.
${}^3$ Partially supported by Fondecyt 1110268 and Mathamsud 13MATH-03.
${}^4$ Partially supported by Basal-CMM-Conicyt, Milenio grant-P05-004F, Mathamsud 13MATH-03 and Fondecyt 1110268.}

\begin{document}
\maketitle
\thispagestyle{empty}
\centerline{\scshape Carmen Cort\'azar$\,{}^1$}\vskip 0.02cm
{\footnotesize
\centerline{Facultad de Matem\'aticas}
\centerline{Pontificia Universidad Cat\'olica de Chile, Casilla 306 Correo 22, Santiago, Chile}
\centerline{email: ccortaza@mat.puc.cl}
}\vskip 0.3cm
\centerline{\scshape Jean Dolbeault$\,{}^2$}\vskip 0.02cm
{\footnotesize
\centerline{Ceremade (UMR CNRS no. 7534), Universit\'e Paris-Dauphine,}
\centerline {Place de Lattre de Tassigny, 75775 Paris C\'edex~16, France}
\centerline{email: dolbeaul@ceremade.dauphine.fr}
}\vskip 0.3cm
\centerline{\scshape Marta Garc\'ia-Huidobro$\,{}^3$}\vskip 0.02cm
{\footnotesize
\centerline{Facultad de Matem\'aticas}
\centerline{Pontificia Universidad Cat\'olica de Chile, Casilla 306 Correo 22, Santiago, Chile}
\centerline{email: mgarcia@mat.puc.cl}
}\vskip 0.3cm
\centerline{\scshape R\'aul Man\'asevich$\,{}^4$}\vskip 0.02cm
{\footnotesize
\centerline{DIM \& CMM (UMR CNRS no. 2071), FCFM,}
\centerline{Universidad de Chile, Casilla 170 Correo 3, Santiago, Chile}
\centerline{email: manasevi@dim.uchile.cl}
\vskip 0.5cm

\centerline{\today}
}\vspace*{0.5cm}

\begin{spacing}{0.9}\begin{quote}{\normalfont\fontsize{10}{12}\selectfont{\bfseries Abstract.} We consider radial solutions of a general elliptic equation involving a weighted $p$-Laplace operator with a subcritical nonlinearity. By a shooting method we prove the existence of solutions with any prescribed number of nodes. The method is based on a change of variables in the phase plane, a very general computation of an angular velocity and new estimates for the decay of an energy associated with an asymptotic Hamiltonian problem. Estimating the rate of decay for the energy requires a sub-criticality condition. The method covers the case of solutions which are not compactly supported or which have compact support. In the last case, we show that the size of the support increases with the number of nodes.}\end{quote}\end{spacing}

\section{introduction}\label{Section:Intro}
\setcounter{equation}{-1}

In this paper we shall consider classical radial sign-changing solutions for a problem of the form
\beq\label{eq0}
\mbox{div}\big(\mathsf a\,|\nabla u|^{p-2}\nabla u\big)+\mathsf b\,f(u)=0\,,\quad\lim_{|x|\to+\infty}u(x)=0
\eeq
where $\mathsf a$ and $\mathsf b$ are two positive, radial, smooth functions defined on $\R^d\setminus\{0\}$, $f\in C(\R)$ satisfying some growth conditions, and $\mbox{div}\big(\mathsf a\,|\nabla u|^{p-2}\nabla u\big)$ denotes a $p-$Laplace operator with weight, $p>1$.
Under appropriate conditions on $\mathsf a$ and $\mathsf b$, this equation can be reduced to
\beq\label{eq1}
\mbox{div}\big(\mathsf q\,|\nabla v|^{p-2}\nabla v\big)+\mathsf q\,f(v)=0\,,\quad\lim_{|x|\to+\infty}v(x)=0\,,
\eeq
on $\R^d$, for some  smooth radial function $\mathsf q$ defined on $\R^d\setminus\{0\}$. Equations of this type have been studied previously
in \cite{pghms}, see also \cite{cfp}, where the existence of nonnegative solutions has been studied.

Radial solutions to~\eqref{eq1} satisfy the problem
\begin{equation}\label{eq2}
\begin{gathered}
\big(q(r)\,\phi_p(v')\big)'+q(r)\,f(v)=0\,,\quad r>0\,,\\
v'(0)=0\,,\quad\lim\limits_{r\to+\infty}v(r)=0\,,
\end{gathered}
\end{equation}
with $q(r)=r^{d-1}\,\mathsf q(x)$, $r=|x|$. Here and henceforth, for any $s\in\R\setminus\{0\}$, $\phi_p(s):=|s|^{p-2}\,s$ and $\phi_p(0)=0$.
We denote by $p'$ the H\"older conjugate exponent of $p$, so that
$\frac 1p+\frac 1{p'}=1$, and observe that $\phi_{p'}\circ\phi_p=\mathrm{Id}$.
Also $\phantom{}'$ denotes the derivative with respect to $r=|x|\geq 0$, $x\in\R^d$, and for radial functions, as it is usual, we shall write $v(x)=v(r)$.

We will be interested only in \emph{classical} solutions of~\eqref{eq2}, \emph{i.e.}, functions $v$ in~$C^1([0,\infty);\R)$ such that $q\,\phi_p(v')$ is in~$C^1((0,\infty);\R)$,  and we will look for solutions satisfying $v(0)>0$.\vs

For the weight $q$ we assume:
\[
q\in 
C^1(\R^+;\R^+)\,,\quad q\ge0\quad\mbox{\emph{and}}\quad q'>0\quad\mbox{on}\;(0,\infty)\,,\leqno{(Q1)}
\]
\[
q'/q\quad\mbox{\emph{is strictly decreasing on}}\;(0,\infty)\,,\leqno{(Q2)}
\]
\[
\mbox{\emph{there exist $C_1>0$ and $C_2>0$ such that}}\;C_1\le\frac{r\,q'(r)}{q(r)}\le C_2\quad\forall\,r\in(0,\infty)\,,\leqno{(Q3)}
\]
\[
\mbox{\emph{for any} $r_0>0$}\,,\quad \lim_{r\to+\infty}\big[h(r+r_0)-h(r)\big]=+\infty\,\leqno{(Q4)}
\]
\emph{where $h$ is defined by}
\beq\label{h}
h(r):=\big(q(r)\big)^{p'}\quad\mbox{\emph{for all }}r\ge0\,.
\eeq
 From $(Q1)$, $q$ is a strictly increasing nonnegative function in $\R^+$, and from $(Q3)$, in the form $C_1/r\le q'/q$ for all $r>0$, after integration over $(s,t)$, $s\le t$, we obtain that $t^{C_1}q(s)\le s^{C_1}q(t)$, hence
it must be that $\lim\limits_{s\to0^+}q(s)=0$ and $\lim\limits_{t\to\infty}q(t)=\infty$. For this reason, we redefine $q$ at $0$ if necessary to have $q(0)=0$ and may assume that $q\in C^0(\R^+_0)$.
Also, from $(Q1)$ we may define $Q(r)=\int_0^r q(t)dt$.

A typical example of such a function $q$ is $q(r)=r^{N-1}$ for some $N>1$ but as we shall see below, much more cases of practical interest are covered. Condition $(Q4)$ is slightly weaker than asking that $\lim_{r\to+\infty}(q(r))^{p'-1}\,q'(r)=+\infty$.

As for the nonlinearity $f$, we assume\vs
\noindent $(f1)$ $f\in C(\R)$, \emph{$f$ is locally Lipschitz on $\R\setminus\{0\}$, with $f(0)=0$,}\vs
\noindent $(f2)$ \emph{there exist $\beta^-<0<\beta^+$ such that $F(s)<0$ for all $s\in(\beta^-,\beta^+)\setminus\{0\}$, $F(\beta^{\pm})=0$, $f(s)>0$ for all $s\in[\beta^+,\infty)$, $f(s)<0$ for all $s\in(-\infty,\beta^-]$,}  and $F(\infty)=F(-\infty)$,  where we have denoted $F(s):=\int_0^s f(t)dt$.\vs
\noindent Finally let us set
\[
 \frac 1N:=\liminf_{r\to0_+}\Big(\frac Qq\Big)'(r)\quad\mbox{and}\quad\mu^*:=\left[\frac1p-\frac1N\right]_+
\]
where $x_+$ denotes the positive part of $x$. From  $(Q3)$ it follows that $N> 1$.

We shall assume the following key \emph{sub-criticality} condition:
\begin{itemize}
\item [$(SC)$] \emph{there exist $\alpha\in(0,1)$, $\mu\ge\mu^*$ and $r_0>0$ such that}
\beq\label{sc1}\mu+\Big(\frac Qq\Big)'-\frac1p\ge 0\quad\mbox{\emph{on}}\;(0,r_0)\eeq
\emph{and}
\beq\label{c0}
\lim_{s\to+\infty}\left[\inf_{s_1,\,s_2\in[\alpha\,s,\,s]}\big(F(s_2)-\mu\,s_2\,f(s_2)\big)\,Q\Bigl(\bigl(\tfrac{(1-\alpha)\,s}{\phi_{p'}(f(s_1))}\bigr)^{1/p'}\Bigr)\right]=+\infty\,.
\eeq
\end{itemize}
\begin{remark}\label{mufinito-p*} It follows from $(Q1)$ and $(Q2)$ that $\big(\tfrac Qq\big)'\ge 0$, hence   $\mu^*\le 1/p$.   Indeed, from the definition of $Q$, $(Q1)$ and $(Q2)$ we have
$$Q(r)=\int_0^r\frac{q(t)}{q'(t)}q'(t)dt\le \frac{q(r)}{q'(r)}\int_0^rq'(t)dt=\frac{(q(r))^2}{q'(r)}$$
implying that $1-\frac{Qq'}{q^2}\ge 0$, that is, $(Q/q)'\ge 0$.
We also observe that there is always some $\mu>\mu^*$ such that  \eqref{sc1}  is satisfied.
 If $p<N$ we notice that $\mu^*=\frac1{p^*}$ where $p^*:=\frac{N\,p}{N-p}$ is the usual critical exponent associated with $N$ when $N=d$ is the dimension and there  are no weights in~\eqref{eq0}.

Condition $(SC)$ is the precise condition that will be required in our proof. However if the limit
\[
\gamma:=\lim_{|s|\to+\infty}\frac{s\,f(s)}{F(s)}-1
\]
exists, then the reader is invited to check that $f(s)\sim|s|^{\gamma-1}\,s$ and $F(s)\sim|s|^{\gamma+1}$ when $|s|\to+\infty$. Assume that $\frac 1N:=\lim_{r\to0_+}\big(\frac Qq\big)'(r)$ is defined and such that $N>p$. If $\gamma<p-1$, we further assume that $q(r)=r^{N-1}$ for any $r>0$. We have
\[
Q\Bigl(\bigl(\tfrac{(1-\alpha)\,s}{\phi_{p'}(f(s))}\bigr)^{1/p'}\Bigr)\sim|s|^{-\frac Np(\,(\gamma+1-p)}\quad{as}\quad|s|\to+\infty\,.
\]
 Then conditions \eqref{sc1} and \eqref{c0} are  satisfied if and only if the much simpler \emph{sub-criticality} condition
\[
\gamma+1<p^*
\]
holds, that is, the standard sub-criticality condition for the $p$-Laplace operator in $\R^d$ with $N=d$.
\end{remark}
\begin{remark}\label{calzo}
 It is worth mentioning that   in \cite{cfp}, where the existence of nonnegative solutions to \eqref{eq2} is proven,  they require a little more from the weight $q$ (see $(q3)$ in \cite{cfp}), namely, that the limit
 $$\lim_{r\to0^+}\frac{rq'(r)}{q(r)}\quad \mbox{exists}.$$
 In this case, also $ \lim_{r\to 0^+}\Bigl(\frac{Q}{q}\Bigr)'(r)$ exists, as
$$ \Bigl(\frac{Q}{q}\Bigr)'(r)=1-\frac{Q(r)}{rq(r)}\frac{rq'(r)}{q(r)}$$
and hence, by L'H\^opital's rule,
 \ben
 \lim_{r\to 0^+}\Bigl(\frac{Q}{q}\Bigr)'(r)=1-\lim_{r\to 0^+}\frac{Q(r)}{rq(r)}\frac{rq'(r)}{q(r)}&=&1-\lim_{r\to 0^+}\frac{rq'(r)}{q(r)}\lim_{r\to 0^+}\frac{Q(r)}{rq(r)}\\
 &=&1-\lim_{r\to 0^+}\frac{rq'(r)}{q(r)}\lim_{r\to 0^+}\frac{q}{rq'(r)+q(r)}\\
 &=&\frac{1}{\lim\limits_{r\to 0^+}\frac{rq'(r)}{q(r)}+1}
 \een
 also exists and their critical exponent is the same as ours.
\end{remark}

We will deal with  existence of solutions with \emph{nodes}, which are defined as the zeros of the solution which are contained in the interior of its support.  As we will see in Remark \ref{extrarem} in Section 3, a solution to \eqref{eq2} can only have a finite number of nodes.

\vs Next we establish our main result.

\begin{thm}\label{Thm:Main} Suppose that assumptions $(Q1)$-$(Q4)$, $(f1)$-$(f2)$ and $(SC)$ are satisfied. Then for any given $k\in\mathbb N_0:=\mathbb N\cup\{0\}$, there exists a solution of \eqref{eq2} with exactly $k$ nodes.\end{thm}

The present paper is organized as follows. In Section~\ref{Section:Examples} we shall first explain how the problem with two independent weights $\mathsf a$ and $\mathsf b$ corresponding to Eq.~\eqref{eq0} can be reduced to the problem with a single weight $\mathsf q$ and illustrate our results with various examples. Also, in that section,   we will  discuss the sub-criticality condition.  In Section~\ref{Section:Properties} we have collected  some preliminary results, including a non-oscillation result which, to our knowledge, is new even for the case $p=2$ without weights. Since we will work in a non-standard framework of sign changing solutions, some care is required concerning results of existence and uniqueness. These  have been collected in Proposition~\ref{Prop:Existence-Uniqueness} and a proof is given in Appendix~\ref{Section:Uniqueness}. Then we adapt the methods developed for the standard case, $\mathsf q=1$, and give short but complete proofs. We will then be able to emphasize the differences of our results with the standard case. The core of our paper is concentrated in Section~\ref{Section:Coordinates} that contains our two key estimates:
\begin{enumerate}
\item[1.] The \emph{Rotation Lemma} (Lemma~\ref{AngularVelocity}) measures the speed of rotation around the origin in the phase space, thus providing an estimate of the number of zeros of the solution. Interestingly, the angular velocity is estimated for finite energy levels, which allows to discard previous restrictions on the growth of the nonlinearity. Here $r$ plays the role of a time variable.
\item[2.] The \emph{Energy Dissipation Lemma} (Lemma~\ref{nw1}) measures the decay rate of the energy: a solution with large energy needs a large interval in $r$ to bring its energy in a range where the Rotation Lemma applies. The sub-criticality condition is essential as shown in Proposition~\ref{Prop:Ex4}.
\end{enumerate}
A remarkable fact is that our method does not distinguish solutions with compact support and solutions supported on the whole line. When solutions are not compactly supported, we only need to discard oscillations; assumption $(H)$ provides a sufficient condition for this, which is probably not optimal. The proof of Theorem~\ref{Thm:Main} is then completed in Section~\ref{Section:Proof}. Finally, in Section~\ref{Section:final}, we establish the estimates which show that the examples of Section~\ref{Section:Examples} are all covered by our results.

Concerning earlier contributions we shall primarily refer to \cite{cghy,dghm} and the references therein. The goal of this paper is to extend the results of \cite{cghy,dghm} to general weights, and to simplify some of the proofs. We refer to \cite{drghm,pghms} for the study of nonnegative solutions in presence of weights. In particular the change of variables which reduces Eq.~\eqref{eq0} to Eq.~\eqref{eq1} has been considered in~\cite{pghms}. As in \cite{cghy}, we will handle simultaneously the solutions with compact support and the other ones. We shall refer to~\cite{GMZ} and to~\cite{cghy} respectively for multiplicity and existence results when $q(r)=r^{N-1}$. Even if our equations have been considered mostly in the case of nonnegative solutions, see for example \cite{Gazzola-Serrin-Tang} for the case with no weights and \cite{cfp}, for the weighted case,  there is still a large literature on sign changing solutions and we refer to the two above mentioned papers for further details and references. Consequences of a possible asymmetry of $F$ are not fully detailed here: see \cite{Fabry-Manasevich02} for more insight. At first reading, it can be assumed that $f$ is odd and $\beta^-=-\beta^+$, although we provide proofs in the general case.

The change of variables of Section~\ref{Section:Coordinates} can be seen as the canonical change of coordinates corresponding either to \hbox{$N=1$} and $f(u)=|u|^{p-2}\,u$, or to the asymptotic Hamiltonian system in the limit $r\to +\infty$, we refer to \cite{Fabry-Manasevich02,Drabek-Girg-Manasevich01,Balabane-Dolbeault-Ounaies01,dghm} for earlier contributions. The role of the critical exponent has been emphasized in~\cite{gkmy} in terms of existence and non-existence for weighted problems.

As we said  before we will generalize and extend some previous results considered in \cite{cghy,dghm}, for the particular case $\mathsf q=1$ in equation~\eqref{eq1} to the more general situation with \emph{weights}. However, our results are not  simple extensions of previously known ones. For example,  a sub-critical growth condition, used previously
 in \cite{Gazzola-Serrin-Tang, cghy},  is generalized here and we prove that   is not only a sufficient condition, but also necessary: see Example~\ref{Ex:Critical} in Section~\ref{Section:Examples}.
 Compared to the existing literature, we mention two new key ingredients in this paper: the computation of the angular velocity in the phase plane is not anymore based on super- or sub-linear growth assumptions on $f$, and the estimate on the size of the nodal domains, which arises from energy estimates, is also new. Notably these energy estimates are valid for compactly supported solutions as well as for solutions supported in the whole space. Throughout this paper, we will establish a number of \emph{qualitative properties} of the support, when the solutions are compactly supported, and of the nodal domains, which are of interest for applications, for instance in astrophysics. Notice that when solutions cannot be compactly supported, which depends on the local behavior of the nonlinearity $f$ near $0$, double zeros occur at infinity; this will be further commented in Section~\ref{Section:final}.\vs

\section{General weights and examples}\label{Section:Examples}

For radial solutions  Eq.~\eqref{eq0}  can be rewritten as
\begin{equation}\label{eq2u}
\begin{gathered}
\big(a(r)\,\phi_p(u')\big)'+b(r)\,f(u)=0\,,\quad r>0\,,\\
\lim_{r\to0_+}a(r)\,\phi_p(u'(r))=0\,,\quad\lim\limits_{r\to+\infty}u(r)=0
\end{gathered}
\end{equation}
where $r=|x|$. We assume  the weights $a(r)=r^{d-1}\,\mathsf a(x)$, $b(r)=r^{d-1}\,\mathsf b(x)$, satisfy the condition
\vs
\noindent $(W1)$   $a$, $b > 0\,\mbox{ in }\R^+$,\ $a$, $b\in 
C^1(\R^+)$,\ $\displaystyle(b/a)^{1/p}\in\mathrm L_{\rm loc}^1(\R^+_0)$ \hbox{and}
$\displaystyle\lim_{r\to+\infty}\int_0^r\(\frac{b(t)}{a(t)}\)^{1/p}dt=+\infty.$
\vs
For all $r>0$, let us define
\[
\chi(r):=\int_0^r\(\frac{b(t)}{a(t)}\)^{1/p}dt.
\]
Then  $\chi:\R^+_0\to\R^+_0$ is such that $\chi(0)=0$, $\lim_{r\to+\infty}\chi(r)=+\infty$, by $(W1)$, and  is a diffeomorphism of $\R^+_0$ onto $\R^+_0$,
with inverse $r=\chi^{-1}(t)$, for any $t\ge0$.

\vs
As in~\cite{pghms}, we introduce the change of variable
\[
t=\chi(r)\quad\forall\,r>0,
\]
and set
\beq
q:=\(a\circ\chi^{-1}\)^{1/p}\,\(b\circ\chi^{-1}\)^{1/p'}.
\label{newq}
\eeq
Then, it is immediate to verify that  $u=u(r)$ is a radial solution of  the equation in \eqref{eq2u} if and only if  $v=u\circ\chi^{-1}$ is a solution of  the equation in \eqref{eq2}  and clearly $\lim\limits_{r\to\infty}u(r)=0$ if and only if $\lim\limits_{t\to\infty}v(t)=0$.
\vs
Next under certain additional conditions on $a$ and $b$ we will  deduce from Theorem~\ref{Thm:Main}  a useful corollary. We begin by definying
\[ \psi:=\(\dfrac1p\,\dfrac{a'}{a}+\dfrac1{p'}\,\dfrac{b'}{b}\)\Bigl(\dfrac{a}{b}\Bigr)^{1/p}\]
and assume that the following conditions hold:
\vs
\noindent$(W2)$ the function $\psi$ is positive and strictly decreasing on $\R^+$,\vs
\noindent$(W3)$ there is $C_1>0$ and $C_2> 0$ such that
\[
C_1\le\chi(r)\,\psi(r)\le C_2\quad\forall\,r>0\,,
\]
\noindent$(W4)$ for any $ r_0>0$,
\[
\lim_{r\to+\infty}\left[\big(a(r+r_0)\big)^{p'-1}\,b(r+ r_0)-\big(a(r)\big)^{p'-1}\,b(r)\right]=+\infty\,.
\]
 Under conditions $(W1)$-$(W4)$ we have that the weight $q$ as defined in (\ref{newq})
 satisfies conditions $(Q1)$-$(Q4)$. Furthermore,  by observing that for $r>0$
 \beq\label{bl1}
 b(r)=\Bigl(\frac{b(r)}{a(r)}\Bigr)^{1/p}a^{1/p}(r)b^{1/p'}(r)\quad\mbox{and}\quad\phi_{p'}\(\frac{B(r)}{a(r)}\)\le \chi^{p'-1}(r)\chi'(r)
 \eeq
 we deduce from $(W1)$ and $(W2)$ that $b\in L^1_{loc}(\R^+_0)$ and $\phi_{p'}\(\frac{B}{a}\)\in L^1_{loc}(\R^+_0)$, hence
 setting
\[
B(r):=\int_0^rb(t)\;dt\quad\mbox{and}\quad\mathcal T_W(z):=\int_0^z\phi_{p'}\(\frac{B(r)}{a(r)}\)\;dr\,,
\]
it can be verified that
\[\label{mur}
\mu^*=\limsup_{r\to0_+}\left[\,\frac1{p\,a}\,\Bigl(\frac{B\,a}b\Bigr)'-\Bigl(\frac Bb\Bigr)'\,\right]
\]
and $(SC)$ reads
\[\leqno{(SC_W)}
\quad\begin{cases}\quad\mbox{\sl there exist $\alpha\in(0,1)$, $\mu\ge\mu^*$ and $r_0>0$ such that}\\
\quad\mu+\bigl(\frac Bb\bigr)'-\frac1{p\,a}\,\bigl(\frac{B\,a}b\bigr)'\ge 0\quad\mbox{\sl for all }r\in(0,r_0)\,,\\
\quad\lim_{s\to+\infty}\left[\inf_{s_1,\,s_2\in[\alpha\,s,\,s]}\big(F(s_2)-\mu\,s_2\,f(s_2)\big)\,B\Bigl(\mathcal T_W^{-1}\bigl(\tfrac{(1-\alpha)\,s}{\phi_{p'}(f(s_1))}\bigr)\Bigr)\right]=+\infty\,.
\end{cases}
\]

Hence from Theorem~\ref{Thm:Main} we obtain the following corollary.
\begin{coro}\label{coro} Suppose that Assumptions $(W1)$-$(W4)$, $(f1)$-$(f2)$, and $(SC_W)$ are satisfied. Then for all given $k\in\N$, there exists a solution of \eqref{eq2u} with exactly $k$ nodes.\end{coro}

We note that when $a=b$ in \eqref{eq2u}, then $\chi(r)=r$ and $\psi(r)=a'(r)/a(r)$. Assumption $(W1)$ reads $a>0$ in $\R^+$ and $a\in C^1(\R^+)$ and assumption $(W2)$ reads $a'>0$ in $\R^+$ and $a'/a$ is strictly decreasing in $\R^+$. Hence the two together are equivalent to $(Q1)$ and $(Q2)$ together with $q=a=b$. Clearly, $(W3)$ and $(W4)$  correspond exactly to $(Q3)$ and $(Q4)$ respectively, with $q=a=b$. Finally, we note that if $u$ is a solution to  \eqref{eq2u} with $u(0)>0$ (the argument is similar if $u(0)<0$), it is necessary (see Proposition \ref{basic1}) that $u(0)>\beta^+$, hence by integration of the equation in \eqref{eq2u} and using the condition $\lim\limits_{r\to0_+}a(r)\,\phi_p(u'(r))=0$ we obtain that as long as $u(r)\ge \beta^+$, say for $r\in(0,r_0)$, it must be that
$$a(r)|u'(r)|^{p-1}=\int_0^r b(s)f(u(s))ds\le C_0B(r),\quad r\in(0,r_0),$$
where we have denoted $C_0=\max\limits_{s\in[\beta^+,u(0)]}f(s)$. Hence, from the second in \eqref{bl1} we obtain
$$|u'(r)|\le \phi_{p'}(C_0)\phi_{p'}\(\frac{B(r)}{a(r)}\)\le \phi_{p'}(C_0)\chi^{p'-1}(r)\chi'(r),\quad r\in(0,r_0),$$
so, when $a=b$, so that $\chi(r)=r$, it holds that
$$|u'(r)|\le \phi_{p'}(C_0)r^{p'-1},\quad r\in(0,r_0),$$
and thus $u'(0)=0$.

\vs

\noindent Many examples fall into the general form of \eqref{eq2u}, that is, of radial solutions to \eqref{eq0}.
The following ones will be examined in more detail in Section \ref{Section:final}.

\begin{example}[Generalized Matukuma equation]\label{ex1}
\[\label{1es}
\begin{gathered}
\Delta_p u + \frac{f(u)}{1+|x|^\sigma}=0\,,\quad x\in\R^d\,,\\
p>1\,, \quad d\ge1\,,\quad\sigma > 0\,.
\end{gathered}
\]
Here $\Delta_p$ denotes the $p$-Laplace operator, namely $\Delta_p u=\nabla\cdot(|\nabla u|^{p-2}\,\nabla u)$, and $a(r)=r^{d-1}$, $b(r)=r^{d-1}/(1+r^\sigma)$, $r=|x|$, $N=d$.
\end{example}

\begin{example}\label{ex2}A second example is provided by the equation
\[\label{2es}
\begin{gathered}
\Delta_p u+\frac{|x|^{\sigma}}{(1+|x|^{p'})^{\sigma/p'}}\,\frac{f(u)}{|x|^{p'}}=0\,,\quad x\in\R^d\,,\\
p>1\,, \quad d\ge1\,,\quad\sigma>0\,,
\end{gathered}
\]
of which Example ~\ref{ex1} is a particular case corresponding to $\sigma=p'$. Now we have $a(r)= r^{d-1}$, $b(r)=r^{d-1+\sigma-p'}/(1+r^{p'})^{\sigma/p'}$. This equation was first introduced in \cite{bfh}, with $p=2$, as a model of stellar structure. Other generalizations of the Matukuma equation and of the stellar model of \cite{bfh} can be found in \cite{cmm,gkmy}.
\end{example}

\begin{example}[$k$-Hessian operator]\label{ex3}
A third example is given by elliptic equations involving the $k$-Hessian operator, $k>0$, see \cite{cmm} which, in case of radial solutions, coincide with
\[\label{2mit}
\begin{gathered}
\nabla\cdot\big(|x|^{1-k}\,|\nabla u|^{k-1}\,\nabla u\big)+f(u)=0\,,\quad x\in\R^d\,,\\
k>0\,,\quad d\ge1\,.
\end{gathered}
\]
Here $p=k+1$, $a(r)=r^{d-k}$, $N=d-k+1$ and $b(r)=r^{d-1}$.
\end{example}

Before starting with proofs, let us comment on the sub-criticality assumptions $(SC)$ and $(SC_W)$. Let us define the energy function by
\beq\label{funct-10}
E(r):=\frac{|v'(r)|^p}{p'}+F(v(r))\,.
\eeq
One of the two main ingredients of our method is based on the fact that for reducing the energy of a solution $v=v_\lambda$ of Eq.~\eqref{eq2} such that $v(0)=\lambda>0$ to a finite, given range, a large interval in $r$ is needed in the large $\lambda$ regime. In other words,
\be\label{defrl}
r_\lambda(a):=\inf\left\{r>0\,:\,E(r)=a\right\}
\ee
is such that
\[
\lim_{\lambda\to+\infty}r_\lambda\big(\theta\,F(\lambda)\big)=+\infty\,,
\]
where $\theta \in (0,1)$ is fixed. To obtain such a property, sub-criticality is essential. The other ingredient then guarantees that the solutions have to change sign a large number of times in the finite energy range. Let us notice that this does not say anything on the intervals delimited by the zeros and does not prevent them, \emph{a priori}, to accumulate at $r=0$. See \cite{GMZ} for results on the number of zeros on a given interval, in the subcritical regime. We will now provide two examples showing that $(SC)$ is needed.

\begin{example}[Critical case]\label{Ex:Critical}
Consider some Aubin-Talenti functions restricted to the interval $r\in[0,R]$ and given by
\[
v(r)=\frac\lambda{\big(1+\frac1{d\,(d-2)}\,\lambda^\frac{4}{d-2}\,r^2\big)^\frac{d-2}2}\quad\forall\,r\in[0,R]\,.
\]
Then $v$ solves Eq.~\eqref{eq2} with $p=2$ on $[0,R]$, $v(0)=\lambda$, $q(r)=r^{d-1}$ and $f(s)=s^{2^*-1}$, $2^*=\frac{2\,d}{d-2}$, for any $s$ in the range $(v(R),\infty)$. We may extend $f$ on $(-\infty,v(R))$  by an odd function satisfying all above conditions except $(SC)$. The reader is invited to check that if $R=R(\lambda)$  is chosen in such a way that $E(R(\lambda))=\sqrt{F(\lambda)}$,   then
\[
\lim_{\lambda\to+\infty}R(\lambda)=0\,,
\]
which clearly indicates that most of the energy is lost in a critical layer as $\lambda\to+\infty$. Let $v=v_\lambda$ be a solution of Eq.~\eqref{eq2} such that $v(0)=\lambda>0$. We will show in Section~\ref{Section:final} that there exists $\lambda_0>\beta^+=-\beta^-$ such that for any $\lambda\ge\lambda_0$ the solution $v_\lambda$ is positive, showing that $(SC)$ is necessary. We have chosen to use $p=2$ for this example, just for simplicity, but computations can be extended to the case of any $p>1$ with no special difficulty.
\end{example}

One may wonder if the property shown in Example~\ref{Ex:Critical} is not directly linked with the scale invariance of the Aubin-Talenti functions and therefore restricted to the critical exponent $\gamma+1=2^*=\frac{2\,d}{d-2}$. The next example shows that this is not the case.

\begin{example}[Slightly super-critical case]\label{Ex:SuperCritical} Consider a radial solution of the Brezis-Nirenberg problem on the unit ball $B\subset\R^d$, $d\ge3$,
\[
-\Delta u=\mu\,u+u^{\frac{d+2}{d-2}+\varepsilon}\quad\mbox{in}\quad B\,,\quad u=0\quad\mbox{on}\quad\partial B
\]
in the slightly super-critical regime, that is, for some $\varepsilon>0$, small. If $\mu>0$ is chosen small enough as $\varepsilon\to0$, then it has been shown in \cite{MR1998635,MR2065022} that the corresponding solution uniformly converges towards an Aubin-Talenti function if $d\ge4$ and, in the radial case, it is easy to deduce that the uniform convergence also holds for the derivative of the solution on the interval $[0,R(\lambda)]$, with $R(\lambda)$ defined as in Example~\ref{Ex:Critical}. Notice that the case $d=3$ deserves a special treatment and has been studied in \cite{MR2103187,MR2091668}.
\end{example}

\begin{example}\label{Ex:Critical2} We conclude our examples by a case where $p^*$ is achieved but $(SC)$ still holds. Let us consider a function $f$ satisfying $(f1)$ and $(f2)$ such that for some $s_0>\max\{2\,e,\,2\,\beta^+\}$ and $\zeta>\frac p{d-p}$, it holds that
\[
f(s)=\frac{|s|^{p^*-2}\,s}{(\log|s|)^\zeta}\quad\mbox{for all $s$ such that } |s|\ge s_0\,,
\]
where as usual $p^*=d\,p/(d-p)$. In case $d\ge 2$ we consider $q(r)=r^{d-1}$, and thus $\mu$ in \eqref{sc1} can be chosen to be $\mu^*=1/p^*$.

This example is inspired by the ground state study of \cite[Theorem 2]{fg}. Our method applies and shows the existence of sign-changing solutions with an arbitrary given number of nodes.
\end{example}

\section{Preliminary results}\label{Section:Properties}

To deal with Problem~\eqref{eq2}, we will use a shooting method and consider the initial value problem
\begin{equation}\label{ivp}
\begin{cases}
\big(q(r)\,\phi_p(v')\big)'+q(r)\,f(v)=0\,,\quad r>0\,,\\
v(0)=\lambda>0\,,\quad v'(0)=0\,.
\end{cases}
\end{equation}
To emphasize the dependence in $\lambda$, we shall denote the solution by $v_\lambda$ whenever necessary.
\begin{prop}\label{Prop:Existence-Uniqueness} Suppose that Assumption $(f1)$-$(f2)$ hold. Assume also that $q$ satisfies $(Q1)$. Then for any fixed $\lambda\in\R$, \eqref{ivp} has a solution~$v_\lambda$ defined in $[0,\infty)$. Moreover, such a solution is unique at least until it reaches a double zero or a point $\overline r$ such that $v_\lambda'(\overline r)=0$ and $f(v_\lambda(\overline r))=0$.\end{prop}
A \emph{double zero} of a function $v$ means some $r>0$ such that $v(r)=v'(r)=0$. This can occur only for values of $v$ or $v'$ for which there is a regularity issue in the equation as, otherwise, the solution would be constant and trivial. Since we are not aware of a reference for the results of Proposition~\ref{Prop:Existence-Uniqueness}, we will sketch a proof. However, this is not central for our paper and requires notations that will be introduced later, so we will postpone it in Appendix~\ref{Section:Uniqueness}.

\medskip The following proposition collects several properties of the solution $v$ to \eqref{ivp}. Properties $(iii)$ and $(iv)$ extend properties that were used in \cite{dghm} and provide us with some understanding of the classification of all solutions. We will come back to this in Section~\ref{Section:final}.
\begin{prop}\label{basic1} Let $v$ be a solution of \eqref{ivp} for some $\lambda>0$, with $q$ and $f$ satisfying $(Q1)$-$(Q3)$ and $(f1)$-$(f2)$ respectively, and consider the energy $E$ defined by \eqref{funct-10}.
\begin{enumerate}
\item[$(i)$] The energy $E$ is nonincreasing and bounded, hence $\lim_{r\to+\infty}E(r)=:\mathcal E$ is finite.
\item[$(ii)$] There exists $C_{\lambda}>0$ such that $|v(r)|+|v'(r)|\le C_{\lambda}$ for all $t\ge 0$.
\item[$(iii)$] If $v$ reaches a double zero at some point $r_0>0$, then $v$ does not change sign on $[r_0,\infty)$. Moreover, if $v\not\equiv 0$ for $r\ge r_0$, then there exists $r_1\ge r_0$ such that $v(r)\neq0$, and $E(r)<0$ for all $r>r_1$ and $v\equiv 0$ on $[r_0,r_1]$.
\item[$(iv)$]If $\lim\limits_{r\to+\infty}v(r)=\ell$ exists, then $\ell$ is a zero of $f$ and $\lim\limits_{r\to \infty}v'(r)=0$.
\end{enumerate}
\end{prop}
\begin{proof} Let $v$ be any solution of \eqref{ivp}. As
\beq\label{I'}
E'(r)=-\,\frac{q'(r)}{q(r)}\,|v'(r)|^p\,,
\eeq
by $(Q1)$, we have that $E$ is decreasing in $r$ implying that $F(v(r))\le E(r)\le F(\lambda)$ for all $r>0$. By $(f2)$, there exists a positive constant $\overline F$ such that $F(s)\ge -\,\overline F$ for all $s\in\R$ and hence $E(r)\ge -\,\overline F$ for all $r\ge0$, thus $(i)$ and $(ii)$ follow.

Assume next that $v$ reaches a double zero at some point $r_0>0$. Then $E(r_0)=0$ implying that $E(r)\le 0$ for all $r\ge r_0$. If $v$ is not constantly equal to $0$ for $r\ge r_0$, then $E(r_1)<0$ for some $r_1>r_0$ and thus, by monotonicity of $E$, $E(r)<0$ for all $r\ge r_1$. Moreover $v$ cannot have the value $0$ again (because at the zeros of $v$ we have $E\ge 0$). This proves $(iii)$ by taking the infimum on all $r_1$ with the above properties.

Finally, assume that $\lim_{r\to+\infty}v(r)=\ell$, and recall that $Q(r):=\int_0^rq(s)\,ds$. From $(Q3)$, we know that $\lim_{r\to+\infty}q(r)=+\infty$ and $\lim_{r\to+\infty}\int_0^r\phi_{p'}(Q/q)\,ds=+\infty$. Using~\eqref{ivp} and applying L'H\^opital's rule twice, we obtain that
\begin{multline*}
0=\lim_{r\to+\infty}\frac{v(r)-\ell}{\int_0^r\phi_{p'}(Q/q)\;ds}= \lim_{r\to+\infty}\frac{v'(r)}{\phi_{p'}(Q/q)(r)}=\phi_{p'}\Bigl(\lim_{r\to+\infty}\frac{q(r)\,\phi_p(v'(r))}{Q(r)}\Bigr)\\
=-\,\phi_{p'}\Bigl(\lim_{r\to+\infty}\frac{q(r)\,f(v(r))}{q(r)}\Bigr)=-\,\phi_{p'}(f(\ell))\,.
\end{multline*}

Next, from the definition of $E$ in~\eqref{funct-10}, it follows that $\lim_{r\to+\infty}|v'(r)|=\big(p'\,(\mathcal E-F(\ell)\big)^{1/p}$ exists. As $\ell=\lim_{r\to+\infty}v(r)$ also exists, we must have that $\lim_{r\to \infty}v'(r)=0$. \end{proof}

\begin{prop}[Asymptotic Hamiltonian system]\label{Pmm1} Let $q$ and $f$ satisfy $(Q1)$-$(Q3)$ and $(f1)$-$(f2)$ respectively, and let $v$ be a solution of \eqref{ivp}. Let $\{r_n\}$ be any sequence in $[0,\infty)$ that tends to $\infty$ as $n\to+\infty$ and define the sequence of real functions $\{v_n\}$ by
\[
v_n(r)=v(r+r_n)\,.
\]
Then $\{v_n\}$ contains a subsequence, still denoted the same, such that $v_n$ and $v_n'$ converge pointwise to a continuous function~$v_\infty$ and $v_\infty'$ respectively, with uniform convergence on compact sets of $[0,\infty)$. Furthermore the function $v_\infty$ is a solution to the asymptotic equation
\be\label{asym12}
\big(\phi_p(v')\big)'+f(v)=0\,,\quad r\in[0,\infty)\,.
\ee
\end{prop}
\begin{proof} We know that there exist two positive constants $c_\lambda^1$ and $c_\lambda^2$ such that
\[
v(r)\leq c_\lambda^1\,,\quad v'(r)\leq c_\lambda^2\,,\quad\text{for all }r\ge 0
\]
and, as a consequence,
\[
v_n(r)\leq c_\lambda^1\,,\quad v_n'(r)\leq c_\lambda^2\,,\quad\text{for all }r\ge 0
\]
for any $n\in\N$. Hence $\{v_n\}$ is equicontinuous: for any $r$, $s>0$ and for all $n\in\N$,
\[
|v_n(s)-v_n(r)| \leq c_\lambda^2\,|s-r|\,.
\]
Then, from Ascoli's theorem (see \cite[Theorem 30]{MR1013117}), $\{v_n\}$ contains a subsequence, still denoted the same, that converges pointwise to a continuous function $v_\infty$, with uniform convergence on compact sets of $[0,\infty)$.

It is clear that each function $v_n$ satisfies
\[
\big(q(r+r_n)\,\phi_p(v_n')\big)'+q(r+r_n)\,f(v_n(r))=0\,,
\]
and hence
\[
\phi_p(v_n'(r))=\frac{q(r_n)}{q(r+r_n)}\,\phi_p(v_n'(0))-\int_0^r\dis\frac{q(s+r_n)}{q(r+r_n)}\,f(v_n(s))\;ds=0\,.
\]
From $(Q3)$, it follows that
\[
\Bigl(\frac{s+r_n}{r+r_n}\Bigr)^{C_1}\le \frac{q(s+r_n)}{q(r+r_n)}\le \Bigl(\frac{s+r_n}{r+r_n}\Bigr)^{C_2}\,,
\]
hence, for any given $r$, $s\ge0$,
\[
\lim_{n\to+\infty}\frac{q(s+r_n)}{q(r+r_n)}=1\,.
\]
By passing to a subsequence if necessary we can assume that $\lim_{n\to \infty}\phi_p(v_n'(0))=a$. Let us choose some $T>0$. Since $\{f(v_n)\}$ uniformly converges in $[0,T]$ to $f(v_\infty)$, we find that $v_n'$ uniformly converges to a continuous function $z$ given by
\[
z(r)=\phi_{p'}\Bigl(a-\int_0^r f(u_\lambda^\infty(s))\;ds\Bigl)\,.
\]
Hence $z'$ exists and is continuous. Furthermore from
\[
v_n(r)=v_n(0) +\int_0^rv'_n(s)\;ds\,,
\]
by letting $n\to+\infty$, we obtain that
\[
v_\infty(r)=v_\infty(0)+\int_0^r z(s)\;ds\,.
\]
Hence $v_\infty$ is continuously differentiable, $v_\infty'(r)=z'(r)$, for all $r\in [0,T]$ and
\[
\phi_p(v_\infty'(r))=a-\int_0^rf(v_\infty(s))\;ds
\]
implies first that $a=\phi_p(v_\infty'(0))$, and then that
\[
(\phi_p(v_\infty'(r)))'+ f(v_\infty(r))=0\,.
\]
This argument shows that $v_\infty$ is a solution to~\eqref{asym12} for all $r\in [0,\infty)$.\end{proof}

\begin{prop}\label{basic3} Let $q$ and $f$ satisfy $(Q1)$-$(Q3)$ and $(f1)$-$(f2)$ respectively, and let $v$ be a solution of \eqref{ivp}. Then $\mathcal E=\lim_{r\to \infty} E(r)=F(\ell)\le0$ and $\ell$ is a zero of $f$.\end{prop}
\begin{proof} Let $T>0$ be arbitrary but fixed. Then
\begin{multline*}
E(k_0\,T)-\mathcal E=\int_{k_0T}^\infty\frac{q'(r)}{q(r)}\,|v'|^p\;dr
=\sum_{k=k_0}^\infty\int_{k\,T}^{(k+1)\,T}\frac{q'(r)}{q(r)}\,|v'|^p\;dr\\
\ge C_1\sum_{k=k_0}^\infty\int_{0}^{T}\frac{|v'(r+k\,T)|^p}{r+k\,T}\;dr
\ge C_1\sum_{k=k_0}^\infty\frac1{(k+1)\,T}\int_{0}^{T}\,|v'(r+k\,T)|^p\;dr
\end{multline*}
where $C_1$ is defined in $(Q3)$. As the left hand side of this inequality is finite, it must be that $\liminf_{k\to+\infty}\int_{0}^{T}\,|v'(r+k\,T)|^p\;dr=0$, hence there is a subsequence $\{n_k\}$ of natural numbers such that $\lim_{k\to+\infty}\int_{0}^{T}\,|v'(r+n_k\,T)|^p\;dr=0$. From Proposition~\ref{Pmm1}, $v_k(r):=v(r+n_k\,T)$ has a subsequence, still denoted the same, such that
\[
\lim_{k\to+\infty}v_k(r)=v_\infty(r)\quad\mbox{and}\quad\lim_{k\to+\infty}v_k'(r)=v_\infty'(r)
\]
uniformly in compact intervals, where $v_\infty$ is a solution of
\[
\big(\phi_p(v_\infty')\big)'+f(v_\infty)=0\,.
\]
Since $\int_0^T|v_\infty'(s)|^p\,ds=0$, $v_\infty$ is a constant, say $v_\infty(r)\equiv \overline v$ and $f(\overline v)=0$. On the other hand,
\[
\frac1{p'}\,|v_k'(r)|^p+F(v_k(r))=\frac1{p'}\,|v'(r+n_k\,T)|^p+F(v(r+n_k\,T))=E(r+n_k\,T)\to\mathcal E
\]
as $k\to+\infty$ and thus, from $(f2)$,
\[
\mathcal E=F(\overline v)\le0\,.
\]
\end{proof}

\begin{remark}\label{rlambda} The monotonicity of $E$ and Proposition \ref{basic3} justify the definition of $r_\lambda(a)$ in \eqref{defrl}. For any $\lambda>\beta^+$, the solution $v_\lambda$ is uniquely defined in $[0,r_\lambda(0))$. Indeed, from Proposition~\ref{Prop:Existence-Uniqueness}, besides double zeros, uniqueness can be lost only at points $r_0$ where $v'(r_0)=0$ and \hbox{$f(v(r_0))=0$} with $v(r_0)\not=0$, and at this kind of point, by $(f2)$, $E(r_0)=F(v(r_0))<0$, hence $r_\lambda(0)<r_0$. Also, $E$ is {\it strictly} decreasing in $[0,r_\lambda(0))$. Finally, we note that if $a>0$, this infimum is a minimum, \emph{i.e.}, it is attained.\end{remark}

We end this section by giving a sufficient condition so that solutions $v_\lambda$ to \eqref{ivp} cannot be oscillatory. This was done in \cite{dghm} under the assumption that $\int_0\frac{ds}{|F(s)|^{1/p}}<\infty$ and without weights. Now we present a result valid for any $f$ satisfying only the structural assumptions $(f1)$-$(f2)$, regardless of the support, but under a slightly stronger assumption than $(Q4)$ on the weight $q$, namely assumption $(H)$ below. We note that this assumption is satisfied for $q(r)=r^{N-1}$ whenever $N>p$, hence it is a true improvement of \cite[Proposition 3.2]{dghm}.
\begin{thm}\label{basic2bis} Let $f$ and $q$ satisfy $(f1)$-$(f2)$, $(Q1)$-$(Q3)$ respectively and assume furthermore that
\begin{itemize}
\item [$(H)$] \emph{Either there exists $\varepsilon>0$ such that $s\mapsto|F(s)|^{-1/p}$ is integrable $(-\varepsilon,\varepsilon)$ or $h$ defined by \eqref{h} is such that $h'$ is non-decreasing and $\lim_{r\to+\infty}h'(r)=+\infty$.}
\end{itemize}
If~$v_\lambda$ solves~\eqref{ivp}, then it has at most a finite number of sign changes. \end{thm}
\begin{proof} From Proposition~\ref{basic3}, we know that $\mathcal E\le0$. From Assumption $(f1)$-$(f2)$, either $\mathcal E<0$ or $\mathcal E=0$ and $\ell=0$, according to Propositions~\ref{basic3} and~\ref{basic1}~(iv).

Assume first that $r_\lambda(0)$ is finite, which corresponds either to $\mathcal E<0$ or to a case in which $\mathcal E=0$ and $v_\lambda$ reaches a double zero at $r_\lambda(0)$. If $\{z_n\}$ is a sequence of zeros accumulating at some $r_\infty\le r_\lambda(0)$, then $v_\lambda(r_\infty)=0$ and for each $n\in\N$, there exists a unique point $r_n\in(z_n,z_{n+1})$ at which $v_\lambda$ reaches a local maximum or minimum value. At these points, using that $E_\lambda(r_n)\ge E_\lambda(z_{n+1})\ge0$, we must have that either $v_\lambda(r_n)\le\beta^-$ or $v_\lambda(r_n)\ge\beta^+$, a contradiction. This proves that $v_\lambda$ has only a finite number of zeros on $(0,r_\lambda(0))$, and by Proposition~\ref{basic1}~$(iii)$, we know that $v_\lambda$ cannot change sign on $(r_\lambda(0),\infty)$.

\medskip The last case corresponds to $\mathcal E=0$ and $\ell=0$ and $r_\lambda(0)=+\infty$. If $s\mapsto|F(s)|^{-1/p}$ is integrable $(-\varepsilon,\varepsilon)$, then the same proof given in \cite[Proposition 3.2]{dghm} can be adapted to the case of a weight $q$ by using $(Q1)$-$(Q3)$ so we omit it. According to $(H)$, we are therefore assuming that $h=q^{p'}$ is such that $h'$ is non-decreasing and $\lim_{r\to+\infty}h'(r)=+\infty$.

We argue by contradiction and suppose that there is an infinite sequence (tending to infinity) of simple zeros of $v$. Then $E_\lambda(r)\ge 0$ for all $r>0$. We denote by $\{z_n^+\}$ the zeros for which $v'(z_n^+)>0$ and by $\{z_n^-\}$ the zeros for which $v'(z_n^-)<0$. We have
\[
0<z_1^-<z_1^+<z_2^-<\cdots<z_n^+<z_{n+1}^-<z_{n+1}^+<\cdots
\]
Between $z_n^-$ and $z_n^+$ there is a minimum $r_n^m$ where $v(r_n^m)<0$ and between $z_n^+$ and $z_{n+1}^-$ there is a maximum $r_n^M$ where $v(r_n^M)>0$. As $E_\lambda(r_n^M),\ E_\lambda(r_n^m)\ge 0$, it must be that $v(r_n^m)<\beta^-$ and $v(r_n^M)>\beta^+$. As we must have $\lim_{r\to+\infty}E_\lambda(r)=0$, it follows that $\lim_{n\to+\infty}v(r_n^M)=\beta^+$ and $\lim_{n\to+\infty}v(r_n^m)= \beta^-$.

Let $b^+$ be the largest positive zero of $f$ ($b^-$ the smallest negative zero of $f$). Set
\[
d^+=\beta^+-b^+\,,\quad b_1=b^++\frac{d^+}4\,,\quad d^-=b^--\beta^-\;,\quad b_2=b^--\frac{d^-}4\,,
\]
and let $a_1$, $a_2$ be such that
\[
b_2<a_2<0<a_1<b_1\,.
\]
We define the unique points $r_{1,n}\in(z_n^+,r_n^M)$, $r_{2,n}\in(r_n^M,z_{n+1}^-)$, $s_{1,n}\in(r_{2,n}$, $z_{n+1}^-)$, $t_{1,n}\in(z_{n+1}^-$, $r_{n+1}^m)$, $s_{2,n}\in (z_{n+1}^-,t_{1,n})$ so that
\[
v(r_{1,n})=b_1=u(r_{2,n})\,,\quad v(s_{1,n})=a_1\,,\quad v(s_{2,n})=a_2\,,\quad v(t_{1,n})=b_2\,.
\]
We have
\[
z_n^+<r_{1,n}<r_n^M<r_{2,n}<s_{1,n}<z_{n+1}^-<s_{2,n}<t_{1,n}<r_{n+1}^m\,.
\]
For $r\in(r_{2,n},s_{1,n})\cup (s_{2,n},t_{1,n})$, $v(r)\in[b_2,a_2]\cup[a_1,b_1]\subset(\beta^-,\beta^+)$, hence $F(v(r))\le0$, but also $|F\!\circ\!v|\ge k_0$ for some positive constant $k_0$ independent of $n$. Moreover, by applying the mean value theorem, we get that there exists a constant $k_1$, which is independent of $n$, such that
\[
0<k_1\le s_{1,n}-r_{2,n}\quad\mbox{and}\quad 0<k_1\le t_{1,n}-s_{2,n}\,.
\]

Next, let us define
\[
\overline f:=\min_{s\in[b_1,\lambda]}f(s)
\]
and notice that $\overline f>0$ by $(f1)$-$(f2)$. From~\eqref{ivp} we have that
\[
|(\phi_p(v'))'(r)|\;=\;\Bigm|\frac{N-1}r\,\phi_p(v'(r))+f(v(r))\Bigm|\;\ge\;\overline f-\frac{N-1}r\,\phi_p(C_{\lambda})
\]
for any $r\in[r_{1,n},r_{2,n}]$. If additionally $r\ge\overline r:=2\,(N-1)\,\phi_p(C_{\lambda})/\,\overline f$, then the r.h.s.~in the above inequality is bounded from below by $\overline f/2$. Hence, choosing $n_0$ such that $z_n^+\ge\overline r$ for all $n\ge n_0$, we have that
\[
|(\phi_p(v'))'(r)|\ge \frac12\,\overline f\quad\mbox{for all }r\in[r_{1,n},r_{2,n}]
\]
and therefore, again from the mean value theorem, we get that
\[
2\,\phi_p(C_{\lambda})\ge |\phi_p(v'(r_{2,n}))-\phi_p(v'(r_{1,n}))|=|(\phi_p(v'))'(\xi)|\,(r_{2,n}-r_{1,n})\ge\frac12\,\overline f\,(r_{2,n}-r_{1,n})
\]
implying that
\[\label{nos1}
r_{2,n}-r_{1,n}\le\frac{2\,\phi_p(C_{\lambda})}{\overline f}\;.
\]
Let
\[\label{H1}
H(r):=h(r)\,E(r)\,.
\]
A straightforward computation shows that
\begin{eqnarray*}
H'(r)&=&p'\,(q(r))^{p'-1}\,q'(r)\,E(r)+h(r)\,E'(r)\nonumber\\
&=&p'\,(q(r))^{p'-1}\,q'(r)\,E(r)-h(r)\,\frac{q'(r)}{q(r)}\,|v'(r)|^p=h'(r)\,F(v(r))\,,\label{Hder1}
\end{eqnarray*}
thus showing that $H'=h'\,(F\!\circ\!v)$. Thus we have
\begin{eqnarray*}\label{nos2}
H(t_{1,n})-H(r_{1,n})=\int_{r_{1,n}}^{t_{1,n}}h'\,(F\!\circ\!v)\;dr&=&\int_{r_{1,n}}^{r_{2,n}}h'\,(F\!\circ\!v)\;dr+
\int_{r_{2,n}}^{t_{1,n}}h'\,(F\!\circ\!v)\;dr\nonumber\\
&=&\int_{r_{1,n}}^{r_{2,n}}h'\,(F\!\circ\!v)\;dr-
\int_{r_{2,n}}^{t_{1,n}}h'\,|F\!\circ\!v|\;dr\nonumber\\
\le \int_{r_{1,n}}^{r_{2,n}}h'\,(F\!\circ\!v)\;dr&-&
\int_{r_{2,n}}^{s_{1,n}}h'\,|F\!\circ\!v|\;dr-\int_{s_{2,n}}^{t_{1,n}}h'\,|F\!\circ\!v|\;dr\nonumber\\
\le \int_{r_{1,n}}^{r_{2,n}}h'\,(F\!\circ\!v)\;dr&-&
k_0\int_{r_{2,n}}^{s_{1,n}}h'\;dr-k_0\int_{s_{2,n}}^{t_{1,n}}h'\;dr\nonumber\\
\le \int_{r_{1,n}}^{r_{2,n}}h'\,(F\!\circ\!v)\;dr&-&2\,k_0\,k_1\,h'(r_{2,n})\,.
\end{eqnarray*}
According to Proposition~\ref{basic3}, $\lim_{n\to+\infty}F(v(r_n^M))=0$. Let us choose $n_0$ large enough so that
\[
\frac{2\,\phi_p(C_{\lambda})}{\overline f}\,F(v(r_n^M))-2\,k_0\,k_1<-\,k_0\,k_1
\]
for all $n\ge n_0$, and hence
\[
\int_{r_{1,n}}^{t_{1,n}}h'\,(F\!\circ\!v)\;dr\le-\,k_0\,k_1\,h'(r_{2,n})\,.
\]
Clearly, we can repeat the above argument in the interval $(t_{1,n},r_{1,n+1})$, thus proving that
\[
H(r_{1,n_0+j})-H(r_{1,n_0})\le-\,k_0\,k_1\sum_{i=0}^{j-1}\Bigl(h'(r_{2,n_0+i})+h'(t_{2,n_0+i})\Bigr)
\]
where $t_{2,n}\in(r_{n+1}^m,z_{n+1}^+)$ is uniquely defined by the condition $u(t_{2,n})=b_2$. Hence
\[
\lim_{j\to+\infty}H(r_{1,n_0+j})=-\infty\,,
\]
implying the contradiction that $E(r_{1,n_0+j})<0$ for some $j$ large enough.
\end{proof}
\begin{remark}\label{extrarem}
A solution $v$ to  problem \eqref{eq2} cannot be oscillatory, that is, it can only have a finite number of nodes. Indeed,  with the notation used in the proof of the above result, $\lim_{n\to+\infty}v(r_n^M)=\beta^+$ and $\lim_{n\to+\infty}v(r_n^m)= \beta^-$, hence $v$ cannot tend to 0 as $r\to\infty$.
\end{remark}

\section{On the number of zeros of solutions to \texorpdfstring{\eqref{ivp}}{ivl}}\label{Section:Coordinates}

In this section, we reformulate the problem in the phase space. We start by computing a lower bound on the angular velocity around the origin.

\medskip Let $v=v_\lambda$ be any solution of \eqref{ivp}. Setting $w=\phi_p(v')$, or equivalently $v'=\phi_{p'}(w)$, Problem~\eqref{ivp} is equivalent to the following first order system.
\beq\label{Flow}
\begin{cases}v'=\phi_{p'}(w)\,,\\
w'=\displaystyle-\,\frac{q'}q\,w-f(v)\,,\\
v(0)=\lambda\,,\quad w(0)=0\,.
\end{cases}
\eeq
We recall that $p'$ stands for the H\"older conjugate of $p$. To the $(v,w)$ coordinates of the phase plane, we assign \emph{generalized polar coordinates} $(\rho,\theta)$ by writing
\[\label{uvxy}
\begin{cases}
v=\rho^{\frac 1p}\cos_{p'}(\theta)\\
w=\rho^{\frac 1{p'}}\sin_{p'}(\theta) \end{cases}
\]
where
\[
\rho=p\,\big[\Phi_p(v)+\Phi_{p'}(w)\big]\,,\quad\mbox{with}\quad\Phi_p(s):=\frac1p\,|s|^p\,,
\]
and $(\cos_{p'}(\theta),\sin_{p'}(\theta))$ is the solution to
\[
\dis\frac{dx}{d\theta}=-\,\phi_{p'}(y)\,,\;\dis\frac{dy}{d\theta}=\phi_p(x)\,,\quad x(0)=1\,,\quad y(0)=0\,.
\]
It is well known, see \cite{DEM}, that solutions to this last system are $2\,\pi_p=2\,\pi_{p'}$ periodic and
\[
\Phi_p\big(\cos_{p'}(\theta)\big)+\Phi_{p'}\big(\sin_{p'}(\theta)\big)=\dis\frac1p\quad\text{for all}\quad\theta\in \R\,.
\]
Notice that in case $p=p'=2$, $(\sqrt\rho,\theta)$ are the usual polar coordinates of $(v,w)$, and $\cos_{p'}$ and $\sin_{p'}$ are the usual $\cos$ and $\sin$ functions. The reader is invited to check that
\[
\sup_{\theta\in[0,2\pi{p'})}\sin_{p'}(\theta)\cos_{p'}(\theta)=\frac 1p\,.
\]

\medskip Now, if $(v,w)$ denotes a solution to~\eqref{Flow} and if we define the corresponding polar functions $r\mapsto\rho(r)$ and $r\mapsto\theta(r)$, then $(\rho,\theta)$ satisfies the following system of equations:
\[\label{Syst:rho-theta}
\begin{cases}
\rho' = p\,\phi_{p'}(w)\,\left[\phi_p(v)-f(v)-\dis\frac{q'}q\, w\right]\,,\\
\theta'=-\frac1\rho\,\left[\,p\,\Phi_{p'}(w)+v\,f(v)+\dis\frac{q'}q\,v\,w \right]\,,\\
\rho(0)=\lambda^p\,,\quad\theta(0)=0\,.\end{cases}
\]
To emphasize the dependance in $\lambda$, we will denote the solution by $(\rho_\lambda,\theta_\lambda)$.

\medskip The following lemma is crucial for the proof of our main result. A similar result was proven in \cite{dghm} for the case of an increasing and superlinear $f$. It establishes a lower bound on the angular velocity $|\theta'|$ around the origin, which will be used to estimate the number of sign changes of $v$ by counting the number of rotations of the solutions around the origin in the phase plane.
\begin{lemma}[\rm Rotation Lemma]\label{AngularVelocity} Let $q$ and $f$ satisfy $(Q1)$-$(Q4)$ and $(f1)$-$(f2)$ respectively. For any $c_0>0$, there exist positive constants $\omega$, $c_1$ and $\overline r$ such that if
\[
\frac12\,c_1\le E(r)\le c_1\quad\mbox{and } r\ge \overline r\,,
\]
then
\[\label{rot}
-\,\theta_\lambda'(r)>\,\omega-\,c_0\,g(v(r),v'(r))
\]
where $g(v,v')=|v\,f(v)\,v'|$ if $\beta^-\le v\le\beta^+$ and $0$ otherwise.\end{lemma}
\begin{proof} Let us choose $c_1$ such that $0<c_1<\lim_{\lambda\to+\infty}F(\lambda)$. It can be verified that
\[
\rho_1\le \rho_\lambda(r)\le \rho_2\,,
\]
where
\ben
\rho_1:&=&\min\left\{\tfrac 14\,p\,c_1,\,\big(F_r^{-1}(c_1/4)\big)^p,|F_\ell^{-1}(c_1/4)|^p\right\}\,,\\
\rho_2:&=&\max\left\{\big(F_r^{-1}(c_1)\big)^p,\,|F_\ell^{-1}(c_1)|^p\}+p\,(c_1+\overline F)\right\}\,,
\een
$F^{-1}_\ell$ denotes the inverse of $F|_{(-\infty,\beta^-]}$ and $F_r^{-1}$ denotes the inverse of $F|_{[\beta^+,\infty)}$. Then we have
\[
-\,\theta'=\(\frac p{p'}\,\frac{|w|^{p'}}\rho+\frac{v\,f(v)}\rho\)+\frac{q'}q\,\frac{v\,w}\rho\ge\(\frac p{p'}\,\frac{|w|^{p'}}\rho+\frac{v\,f(v)}\rho\)-\frac{C_2}{p\,r}\label{mv}
\]
where $C_2>0$ is as in $(Q3)$.

Now, if $v\le\beta^-$ or $v\ge\beta^+$, then
\[
-\,\theta'(r)\ge \frac{A}\rho-\frac{C_2}{p\,r}\ge\frac{A}{\rho_2}-\frac{C_2}{p\,r}
\]
where $A:=\inf\left\{s\,f(s)\ |\ s\in[-\rho_2^{1/p},\beta^-]\cup[\,\beta^+,\rho_2^{1/p}\,]\right\}$.

\medskip Assume next that $\beta^-\le v\le\beta^+$. As $F\!\circ\!v\le 0$, then from $E(r)\ge c_1/2$ we have that $\frac p{p'}\,|w|^{p'}\ge p\,c_1/2$,
\[
-\,\theta'(r)\ge \frac{p\,c_1}{2\,\rho}+\frac{v\,f(v)}\rho-\frac{C_2}{p\,r}\ge \frac{p\,c_1}{2\,\rho_2}-\Bigl(\frac2{p'\,c_1}\Bigr)^{1/p}\frac1{\rho_1}\,|v\,f(v)|\,|v'|-\frac{C_2}{p\,r}\,.
\]
Hence, the conclusion holds with
\[
\omega=\min\Bigl\{\frac{A}{2\,\rho_2},\frac{p\,c_1}{4\,\rho_2}\Bigr\},\quad c_0=\Bigl(\frac2{p'\,c_1}\Bigr)^{1/p}\frac1{\rho_1}\quad\mbox{and}\quad\overline r=\max\Bigl\{\frac{2\,C_2\,\rho_2}{p\,A},\frac{4\,C_2\,\rho_2}{p^2\,c_1}\Bigr\}\,.
\]
\end{proof}

From Proposition~\ref{basic3}, given $c_1$ as above and $\lambda>\beta^+$ such that $c_1<F(\lambda)$, there exist points $ r_\lambda(c_1)<r_\lambda(c_1/2)$, where $r_\lambda(a)$ is defined in \eqref{defrl}. Now we can state the second main ingredient of this paper, which strongly relies on the sub-criticality assumption $(SC)$ as was emphasized in Section~\ref{Section:Examples}.
\begin{lemma}[\rm Energy Dissipation Lemma]\label{nw1} Let $q$ and $f$ satisfy $(Q1)$-$(Q3)$, $(f1)$-$(f2)$ respectively and let $(SC)$ hold. With the above notations, we get
\[
\lim_{\lambda\to+\infty} r_\lambda(c_1)=+\infty\quad\mbox{and}\quad\lim_{\lambda\to+\infty}(r_\lambda(c_1/2)-r_\lambda(c_1))=+\infty\,.
\]
\end{lemma}
\begin{proof} Assume that there exists a sequence $\lambda_n\to+\infty$ and $K>0$ such that $r_{\lambda_n}(c_1)\le K$ for all $n$. Let $\alpha\in(0,1)$, $r_0>0$ and $\mu$ be as in $(SC)$ and set $\mathsf r_\alpha:=\inf\{r>0\,:\,v_\lambda(r)=\alpha\,\lambda\}$. Clearly, we may assume that $K>r_0$ and also $K>\mathsf r_\alpha$, at least for $\lambda$ large.

Let us recall that $E=|v'|^p/p'+F\circ v$ and $E'=-\,q'(r)\,|v'(r)|^p/q(r)$, according to \eqref{funct-10} and~\eqref{I'}, $q\,(Q/q)'=q-\,Q\,q'/q$ and $1/p+1/p'=1$, so that
\[
\frac d{dr}\big(Q\,E+\mu\,q\,v_\lambda\,\phi_p(v_\lambda')\big)=q\,|v_\lambda'|^p\(\mu+\Bigl(\frac Qq\Bigr)'-\frac 1p\)+q\,\Bigl(F(v_\lambda)-\mu\,v_\lambda\,f(v_\lambda)\Bigr)\,.
\]
Then, for $r\ge r_0$ we have
\begin{eqnarray*}
&&\hspace*{-1cm}Q(r)\,E(r)+\mu\,q(r)\,v_\lambda(r)\,\phi_p(v_\lambda'(r))\\
&&\ge\int_0^{r_0}q\underbrace{\Bigl(\mu+\Bigl(\frac Qq\Bigr)'-\frac1p\Bigr)}_{\ge\,0\mbox{ by \eqref{sc1} in $(SC)$}}\,|v_\lambda'|^p\,dt\\
&&\hspace*{5cm}-\frac1p\int_{r_0}^rq\,|v_\lambda'|^p\,dt+\int_0^rq\,\Bigl(F(v_\lambda)-\mu\,v_\lambda\,f(v_\lambda)\Bigr)\;dt\\
&&\ge -\frac1p\,\frac{q^2(r)}{q'(r)}\int_{r_0}^r\frac{q'}{q}\,|v_\lambda'|^p\,dt+\int_0^rq\,\Bigl(F(v_\lambda)-\mu\,v_\lambda\,f(v_\lambda)\Bigr)\;dt\\
&&\hspace*{2cm}=\frac1p\,\frac{q^2(r)}{q'(r)}\,\big(E(r)-E(r_0)\big)+\int_0^rq\,\Bigl(F(v_\lambda)-\mu\,v_\lambda\,f(v_\lambda)\Bigr)\;dt\,.\label{ultacot}
\end{eqnarray*}
With $h:=q^{p'}$ according to \eqref{h}, since $(h\,E)'=h'\,(F\circ v)\ge -\,h'\,\overline F$ so that
\[
(h\,E)(r)-(h\,E)(r_0)\ge-\,\overline F\,\big(h(r)-h(r_0)\big)\,,
\]
we have
\[
E(r_0)\le\frac{h(r)}{h(r_0)}\,E(r)+\overline F\,\Bigl(\frac{h(r)}{h(r_0)}-1\Bigr)\,.
\]
Also recall that $w_\lambda=\phi_p(v_\lambda')$ and so, by Young's inequality,
\[
v_\lambda\,\phi_p(v_\lambda')=v_\lambda\,w_\lambda\le\frac 1p\,|v_\lambda\,|^p+\frac1{p'}\,|w_\lambda\,|^p\,.
\]
Collecting these estimates, we find that
\begin{multline*}
C(r)\,|E(r)|+\frac\mu p\,q(r)\,\rho_\lambda(r)+D(r)\\
\ge Q(r)\,E(r)+\mu\,q(r)\,v_\lambda(r)\,\phi_p(v_\lambda'(r))-\,\frac1p\,\frac{q^2(r)}{q'(r)}\,\big(E(r)-E(r_0)\big)\\
\ge\int_0^rq(t)\Bigl(F(v_\lambda(t))-\mu\,v_\lambda(t) f(v_\lambda(t))\Bigr)\;dt
\end{multline*}
with
\[
C(r):=Q(r)+\frac1p\,\frac{q^2(r)}{q'(r)}\,\Bigl(\frac{h(r)}{h(r_0)}-1\Bigr)\quad\mbox{and}\quad D(r):=\frac1p\,\frac{q^2(r)}{q'(r)}\,\overline F\,\Bigl(\frac{h(r)}{h(r_0)}-1\Bigr)\,.
\]
{}From $(SC)$, there exists $\overline M>0$ such that
\[
F(s)-\mu\,s\,f(s)\ge -\,\overline M\quad\mbox{for all }s\in\R\,.
\]
Let us choose some $R>K$ and define $\overline C$ and $\overline D$ as the maxima of $C$ and $D$ on $[0,R]$ respectively. For any $r\in[0,R]=[0,\mathsf r_\alpha]\cup[\mathsf r_\alpha,R]$, we have
\begin{multline*}
\overline C\,|E(r)|+\frac\mu p\,q(R)\,\rho_\lambda(r)+\overline D\\
\ge\int_0^{\mathsf r_\alpha}q(t)\,\Bigl(F(v_\lambda(t))-\mu\,v_\lambda(t)\,f(v_\lambda(t))\Bigr)\;dt-\,\overline M\,\big(Q(R)-Q(\mathsf r_\alpha)\big)
\end{multline*}
and thus
\ben
\overline C\,|E(r)|+\frac\mu p\,q(R)\,\rho_\lambda(r)+\overline D\ge\Bigl(F(s_2)-\mu\,s_2\,f(s_2)\Bigr)\,Q(\mathsf r_\alpha)-\,\overline M\,Q(R)\,,
\een
where we have set $s_2:=\mathrm{Argmin}\{F(s)-\mu\,s\,f(s)\,:\,s\in[\alpha\,\lambda,\lambda]\}$. If $s_1\in[\alpha\,\lambda,\lambda]$ is such that $f(s_1)=\max\limits_{s\in[\alpha\,\lambda,\lambda]}f(s)$. From~\eqref{ivp} we get
\ben
-\,q(r)\,\phi_p(v'_\lambda(r))=\int_{0}^rq(t)f(v_\lambda(t))\;dt\le f(s_1)\,Q(r)
\een
as long as $0<r<\mathsf r_\alpha$, hence
\ben
-\,v'_\lambda(r)\le \phi_{p'}(f(s_1))\,\phi_{p'}\Bigl(\frac{Q(r)}{q(r)}\Bigr)\quad\mbox{for any}\;r\in[0,\mathsf r_\alpha]\,.
\een
Integrating now this last inequality over $[0,\mathsf r_\alpha]$ we obtain
\[
(1-\alpha)\,\lambda\le \phi_{p'}(f(s_1))\int_0^{\mathsf r_\alpha}\phi_{p'}\Bigl(\frac{Q(r)}{q(r)}\Bigr)\;dr\,.
\]
From assumption $(Q3)$, it follows that
\begin{multline*}
\frac1{C_2}\(r\,q(r)-Q(r)\)=\frac1{C_2}\int_0^rt\,q'(t)\;dt=\frac1{C_2}\int_0^r\frac{t\,q'(t)}{q(t)}q(t)\;dt\le Q(r)\\
=\int_0^r\frac{q(t)}{q'(t)}\,q'(t)\;dt\le\frac1{C_1}\int_0^rt\,q'(t)\;dt=\frac1{C_1}\(r\,q(r)-Q(r)\),
\end{multline*}
so that there exist positive constants $D_1$ and $D_2$ such that
\beq\label{Qq}
D_1\le\frac{r\,q(r)}{Q(r)}\le D_2\quad\mbox{for all}\;r>0\,,
\eeq
and hence $t\mapsto t^{1-p'}\,\phi_{p'}\Bigl(\frac{Q(t)}{q(t)}\Bigr)$ is bounded away from $0$ and from above by two positive constants. This implies that there exists a positive constant $C_0=C_0(D_1,D_2)$ such that
\[
\mathsf r_\alpha\ge C_0\Bigl(\frac{(1-\alpha)\,\lambda}{\phi_{p'}(f(s_1))}\Bigr)^{1/p'}\,.
\]
Therefore, using \eqref{Qq} we get, for yet another constant $\overline C_0$, that
\ben
\lefteqn{\overline C\,|E(r)|+\frac\mu p\,q(R)\,\rho_\lambda(r)+\overline D\qquad\qquad}\\
&&\ge\overline C_0\,\inf_{s_1,s_2\in[\alpha\,\lambda,\lambda]}\Bigl(F(s_2)-\mu\,s_2\,f(s_2)\Bigr)\,Q\Bigl(\Bigl(\frac{(1-\alpha)\,\lambda}{\phi_{p'}(f(s_1))}\Bigr)^{1/p'} \Bigr)-\,\overline M\,Q(R)\,,
\een
implying by \eqref{c0} in $(SC)$ that
\[
\lim_{n\to+\infty}\(\,\left|\,\frac1{p'}\,|v_{\lambda_n}'(r)|^p+(F\circ v_{\lambda_n})(r)\right|+\frac\mu p\,q(R)\,\rho_{\lambda_n}(r)\)=+\infty
\]
uniformly in $[0,R]$. Therefore, $\lim_{n\to+\infty}\rho_{\lambda_n}(r)=+\infty$ uniformly in $[K,R]$, a contradiction.

The second assertion follows from the first, by noting that from the mean value theorem applied to $E$ in $[r_\lambda(c_1),r_\lambda(c_1/2)]$ and $(Q3)$,
\[
\frac12\,c_1\le p'\,\frac{C_2\,(c_1+\overline F)}{r_{\lambda}(c_1)}\,\big(r_{\lambda}(c_1/2)-r_{\lambda}(c_1)\big)\,.
\]\end{proof}

Now we start to make use of the variables introduced in the beginning of this section in order to estimate the number of sign changes of the solutions. Here $[x]$ denotes the integer part of $x$.
\begin{lemma}\label{Increasing} For any $R>0$, the number of nodes of~$v_\lambda$ in $(0,R)$ is greater or equal than
\[
\lim_{r\to R_-}\left[\frac{|\theta_\lambda(r)|}{\pi_p}-\,\frac32\right]
\]
and larger or equal than
\[
\lim_{r\to R_+}\left[\frac{|\theta_\lambda(r)|}{\pi_p}-\,\frac12\right]\,.
\]
\end{lemma}
\begin{proof} These estimates follow directly from the change of variables $(v,w)\mapsto(\rho,\theta)$. Some care is required in case a zero corresponds to the boundary of the support of the solution. Here we assume that the support of the solution is simply connected otherwise the lower bound on the number of nodes has to be decreased by one unit.\end{proof}

We are now in a position to prove a result concerning the number $N(\lambda)$ of nodes of the solution to \eqref{ivp} with initial value $\lambda$. This result is the core argument of the proof of Theorem~\ref{Thm:Main}.
\begin{prop}\label{lambdaBig} Let $q$ and $f$ satisfy $(Q1)$-$(Q3)$, $(f1)$-$(f2)$ respectively and let $(SC)$ hold. Then
\[\lim_{\lambda\to+\infty}N(\lambda)=+\infty\,.\]
\end{prop}
\begin{proof} By Lemma~\ref{AngularVelocity},
\ben
-\,\theta(r_{\lambda}(c_1/2))+\theta(r_{\lambda}(c_1))&\ge& \omega\,\big(r_{\lambda}(c_1/2)-r_{\lambda}(c_1)\big)-c_0\int_{r_{\lambda}(c_1)}^{r_{\lambda}(c_1/2)}g(v(t),v'(t))\;dt\\
&\ge& \omega\,\big(r_{\lambda}(c_1/2)-r_{\lambda}(c_1)\big)-c_0\,\big(N^\lambda_{[r_{\lambda}(c_1),r_{\lambda}(c_1/2)]}+2\,\big)\,G
\een
where $G=\int_{\beta^-}^{\beta^+}\,|s\,f(s)|\,ds$ and $N^\lambda_{[r_{\lambda}(c_1),r_{\lambda}(c_1/2)]}$ is the number of zeros of $v$ in $[r_{\lambda}(c_1),r_{\lambda}(c_1/2)]$. Since $\lim_{\lambda\to+\infty}\big(r_{\lambda}(c_1/2)-r_{\lambda}(c_1)\big)=+\infty$ by Lemma~\ref{nw1} and since
\[
\pi_p\,N^\lambda_{[r_{\lambda}(c_1),r_{\lambda}(c_1/2)]}\ge-\,\theta(r_{\lambda}(c_1/2))+\theta(r_{\lambda}(c_1))
\]
according to Lemma~\ref{Increasing}, we have shown that $N(\lambda)\ge N^\lambda_{[r_{\lambda}(c_1),r_{\lambda}(c_1/2)]}\to+\infty$ as $\lambda\to+\infty$.\end{proof}

\section{Proof of \texorpdfstring{Theorem~\ref{Thm:Main}}{Theorem 1.1}}\label{Section:Proof}

For $k\in\N_0:=\N\cup\{0\}$ we define the sets
\[\label{ak}
A_k:=\{\lambda\ge \beta^+\ :\ (v_\lambda(r),w_\lambda(r))\not=(0,0)\quad\mbox{for all }r\ge0\,,\mbox{ and } N(\lambda)=k\}\,,
\]
\[\label{ik}
I_k:=\{\lambda\ge \beta^+\ :\ (v_\lambda(r_\lambda(0)),w_\lambda(r_\lambda(0)))=(0,0)\quad\mbox{and } N_{[0,r_\lambda(0))}(\lambda)=k\}\,.
\]
Recall that $r_\lambda(a):=\inf\{r\geq 0\,:\,E(r)=a\}$, and $E$ has been defined by~\eqref{funct-10}. Notice that $r_\lambda(0)$ can be finite or infinite. We have
\[
[\beta^+,\infty)=\big(\cup_{k\in\N_0}I_k\big)\cup\big(\cup_{k\in\N_0}A_k\big)\cup A_\infty\,,
\]
where we have denoted by $A_\infty:=\{\lambda\ge\beta^+\ |\ v_\lambda \mbox{ is  oscillatory}\}$.  Notice that $A_\infty=\emptyset$ if~$(H)$ holds. Indeed, let $\lambda\ge \beta^+$, $\lambda\not\in A_\infty$. Then $N(\lambda)=j$ for some $j\in\N_0$. If $v_\lambda(r_\lambda(0))\not=0$, then $v_\lambda$ does not have any double zero in $[0,\infty)$. Indeed, assume by contradiction that $r_1>r_\lambda(0)$ is a double zero of $v_\lambda$. Then by the monotonicity of $E$, $E(r)\equiv 0$ in $[r_\lambda(0),r_1]$. But then also $E'(r)\equiv 0$ in $(r_\lambda(0),r_1)$ implying that $v_\lambda'(r)\equiv 0$ in $(r_\lambda(0),r_1)$ and thus $v_\lambda(r_\lambda(0))=v_\lambda(r_1)=0$, a contradiction. Hence $\lambda\in A_j$. If $v_\lambda(r_\lambda(0))=0$, then by the definition of $r_\lambda(0)$ we also have $v_\lambda'(r_\lambda(0))=0$ hence $\lambda\in I_j$. Also, observe that the sets $A_i$, $I_j$ are disjoint for any $i$, $j$, and for $i\not=j$, $A_i\cap A_j=\emptyset$ and $I_i\cap I_j=\emptyset$.

\begin{prop} Let $q$ and $f$ satisfy $(Q1)$-$(Q4)$, $(f1)$-$(f2)$ respectively and let $(SC)$ hold. With the above notations, we have:
\begin{itemize}
\item[$(i)$] $A_k$ is open in $[\beta^+,\infty)$,
\item[$(ii)$] $A_k\cup I_k$ is bounded,
\item[$(iii)$] if $\lambda_0\in I_k$, then there exists $\delta>0$ such that $(\lambda_0-\delta,\lambda_0+\delta)\subset A_k\cup A_{k+1}\cup I_k$,
\item[$(iv)$] given $k\in\mathbb N$ and $\lambda_0\in A_\infty$, there exists $\delta>0$ such that $(\lambda_0-\delta,\lambda_0+\delta)\cap A_k=\emptyset$ and $(\lambda_0-\delta,\lambda_0+\delta)\cap I_k=\emptyset$,
\item[$(v)$] $\sup A_k\in I_{k-1}\cup I_k$, where we set $I_{-1}=\emptyset$,
\item[$(vi)$] $\sup I_k\in I_k$.
\end{itemize}
\end{prop}
\begin{proof}
\noindent $(i)$ $A_k$ is open in $[\beta^+,\infty)$: Indeed, if $\overline\lambda\in A_k$, then in particular $(v_{\overline\lambda}(\overline r),w_{\overline\lambda}(\overline r))\not=(0,0)$, where $\overline r=r_{\overline\lambda}(0)$. Then there exists $\varepsilon_0>0$ such that the solution of \eqref{ivp} is unique in $[0,r_{\overline\lambda}(0)+\varepsilon]$ and $E_{\overline\lambda}(r_{\overline\lambda}(0)+\varepsilon/2)<0$ for all $\varepsilon\in(0,\varepsilon_0]$, and thus there exists $\delta>0$ such that
\[
E_{\lambda}(r_{\overline\lambda}(0)+\varepsilon/2)<0
\]
for all $\lambda\in(\overline\lambda-\delta,\overline\lambda+\delta)$ implying that $r_\lambda(0)\le r_{\overline\lambda}(0)+\varepsilon/2$. On the other hand, for the same reason, there exists $\delta'>0$ such that
\[
E_{\lambda}(r_{\overline\lambda}(0)-\varepsilon/2)>0
\]
for all $\lambda\in(\overline\lambda-\delta',\overline\lambda+\delta')$ implying that $r_\lambda(0)\ge r_{\overline\lambda}(0)-\varepsilon/2$. We conclude then that $\lim_{\lambda\to\overline\lambda}r_\lambda(0)=r_{\overline\lambda}(0)$. Hence the openness of $A_k$ follows from the continuous dependence of the solutions in the initial value $\lambda$.

\medskip\noindent $(ii)$ The boundedness of $A_k\cup I_k$ is a consequence of Proposition ~\ref{lambdaBig}.
\medskip

\noindent
$(iii)$ Let $\lambda_0\in I_k$, set $r_0=r_{\lambda_0}(0)$ and let
\[
0<z_{1,0}<z_{2,0}<\ldots<z_{k,0}<r_0
\]
denote the $k$ zeros of $v_{\lambda_0}$ in $(0,r_0)$.

Assume first that $v_{\lambda_0}$ is decreasing in $(r_0-2\,\varepsilon_0,r_0)$ for some $\varepsilon_0>0$, so that it reaches a last maximum point at some $s_{k,0}\in(z_{k,0}, r_0)$. Let us define $H_\lambda:=h\,E_\lambda$ where $h=q^{p'}$, $E_\lambda=E$ defined by \eqref{funct-10} when $v=v_\lambda$ is a solution to \eqref{ivp}, and recall that $H_\lambda'=h'\,(F\circ v_\lambda)$.

As $\lim_{r\to r_0}v_{\lambda_0}(r)=0$, we have that $H_\lambda$ is decreasing in a left neighborhood of $r_0$ and thus $\lim_{r\to r_0}H_{\lambda_0}(r)=L\ge0$. Then, given $\varepsilon>0$, there exists $\overline r<r_0$ such that
\[
0<v_{\lambda_0}(\overline r)<\frac12\,\beta^+\quad\mbox{and}\quad H_{\lambda_0}(\overline r)<L+\varepsilon\,.
\]
Hence by continuous dependence of solutions to \eqref{ivp} in the initial data in any compact subset of $[0,r_0)$, there exists $\delta_0>0$ such that for $\lambda\in(\lambda_0-\delta_0,\lambda_0+\delta_0)$, the solution~$v_\lambda$ satisfies
\be\label{mm2}
0<v_\lambda(\overline r)<\beta^+\,,\quad H_\lambda(\overline r)<L+2\,\varepsilon\quad\mbox{and $v_\lambda$ has at least $k$ simple zeros in $[0,r_0)$}\,,
\ee
that is,
\[\label{mm1}
(\lambda_0-\delta_0,\lambda_0+\delta_0)\subset\big(\cup_{j\ge k}A_j\big)\cup\big(\cup_{j\ge k}I_j\big)\cup A_\infty\,.
\]
Now we argue by contradiction and assume that there is a sequence $\{\lambda_n\}$ converging to $\lambda_0$ as $n\to+\infty$ such that $\lambda_n\not\in A_k\cup A_{k+1}\cup I_k$. Hence we have\[
\lambda_n\in\big(\cup_{j\ge k+2}A_j\big)\cup\big(\cup_{j\ge k+1}I_j\big)\cup A_\infty\,,
\]
that is, the solution $v_{\lambda_n}$ has at least $k+2$ zeros and at least the first $k+1$ zeros are simple. Let us denote these zeros by
\[
0<z_{1,n}<z_{2,n}<\ldots<z_{k,n}<z_{k+1,n}<z_{k+2,n}\,.
\]
By the choice of $\overline r$ and \eqref{mm2}, $v_{\lambda_n}$ decreases in $[\overline r, z_{k+1,n}]$. Let us denote by $s_{k+1,n}$ the point in $(z_{k+1,n},z_{k+2,n})$ where $v_{\lambda_n}$ reaches its minimum value. As $E_{\lambda_n}(z_{k+2,n})\ge0$, we must have that
\[
v_{\lambda_n}(s_{k+1,n})<\beta^-\,.
\]
Let us denote by $r_{1,n}<r_{2,n}$ the unique points in $(z_{k+1,n},s_{k+1,n})$ where
\[
v_{\lambda_n}(r_{1,n})=\frac14\,\beta^-\quad\mbox{and}\quad v_{\lambda_n}(r_{2,n})=\frac12\,\beta^-\,.
\]
{}From the mean value theorem we have that
\beq\label{sep}
\frac14\,|\beta^-|=|v_{\lambda_n}(r_{2,n})-v_{\lambda_n}(r_{1,n})|\le (C_{\lambda_0}+1)\,(r_{2,n}-r_{1,n})
\eeq
for $n$ large enough, where $C_{\lambda_0}$ has been defined in Proposition~\ref{basic1}~(ii), and
\[
\frac{\lambda_0}2\le \lambda_n=|v_{\lambda_n}(0)-v_{\lambda_n}(z_{1,n})|\le (C_{\lambda_0}+1)\,z_{1,n}\le (C_{\lambda_0}+1)\,r_{1,n}\,,
\]
hence $r_{1,n}$ is bounded below uniformly by a positive constant $c_0:=\frac12\,\lambda_0/(C_{\lambda_0}+1)$.

From $H_\lambda'=h'\,(F\circ v_\lambda)$, and using the first estimate in \eqref{mm2}, we have that for $n$ large enough, $H'_{\lambda_n}( r)<0$ for $r\in[\overline r,z_{k+1,n}]$ and thus by the second estimate in \eqref{mm2}, we know that $H_{\lambda_n}(z_{k+1,n})<L+2\,\varepsilon$. Integrating now $H_\lambda'=h'\,(F\circ v_\lambda)$ over $[z_{k+1,n},r_{2,n}]$, and using that $F\circ v_{\lambda_n}<0$ in this range we find that
\[
H_{\lambda_n}(r_{2,n})-H_{\lambda_n}(z_{k+1,n})=-\int_{z_{k+1,n}}^{r_{2,n}}h'\,|F\circ v_{\lambda_n}|\;dt\le-\,C\,\big(h(r_{2,n})-h(r_{1,n})\big)
\]
where $C:=\inf_{s\in[\frac12\,\beta^-,\frac14\,\beta^-]}\,|F(s)|$. Hence, using the monotonicity of $q$ and \eqref{sep}, we obtain
\[\label{c13}
H_{\lambda_n}(r_{2,n})\le L+2\,\varepsilon-\,C\,\big[h(r_{1,n}+\overline c_0)-h(r_{1,n})\big]
\]
where $\overline c_0:=\frac14\,|\beta^-|/(C_{\lambda_0}+1)$.

Now we need to distinguish the cases $r_0=+\infty$ and $r_0<+\infty$. In the first case, it must be that $r_{1,n}\to+\infty$ as $n\to+\infty$. Indeed, if some subsequence $\{r_{1,k_n}\}$ is bounded, say if $r_{1,k_n}\le K$ for any $n\in\N$, then by the continuous dependence of solutions in $[0,2K]$ we obtain a contradiction. Therefore, from $(Q4)$ we obtain that $\lim_{n\to+\infty}\big[h(r_{1,n}+\overline c_0)-h(r_{1,n})\big]=+\infty$ and thus, for $n$ large enough, $H_{\lambda_n}(r_{2,n})<0$, contradicting the fact that $E_{\lambda_n}(r_{2,n})\ge E_{\lambda_n}(z_{k+2,n})\ge 0$.

To analyze the second case, \emph{i.e.}~the case $r_0<+\infty$, we first observe that from $(Q1)$ and~$(Q4)$, there exists a positive constant $a_0$ such that
\[
h(y+\overline c_0)-h(y)\ge a_0\quad\mbox{for all }y\in[c_0,\infty)\,.
\]
Since it holds that $L=0$, by choosing $\varepsilon<C\,a_0/2$ we again obtain that $H_{\lambda_n}(r_{2,n})<0$, a contradiction.

The case in which $v_{\lambda_0}$ is increasing in a left neighborhood of $r_0$ can be handled with similar arguments. Altogether $(iii)$ is established.

\medskip\noindent $(iv)$ Let $k\in\mathbb N$ and $\lambda_0\in A_\infty$. Then there exists $R>0$ such that $v_{\lambda_0}$ has at least $k+2$ zeros in $[0,R]$. By continuous dependence, there exists $\delta>0$ such that for $\lambda\in(\lambda_0-\delta,\lambda_0+\delta)$, $v_\lambda$ has at least $k+1$ zeros in $[0,R]$, thus the result follows.

\medskip\noindent $(v)$ Assume next that $A_k\not=\emptyset$, let $\lambda_0=\sup A_k$ and set $r_0=r_{\lambda_0}(0)$. By $(iv)$ and using that $A_j$ is open for every $j\in\N_0$, $\lambda_0\not\in A_j$ for any $j$ hence $\lambda_0\in I_j$ for some $j$, and by continuous dependence of the solutions in the initial data in $[0,r_0-\varepsilon]$ for $\varepsilon>0$ small enough, $j\le k$. By $(iii)$, there is $\delta>0$ such that $(\lambda_0-\delta,\lambda_0]\subset A_j\cup A_{j+1}\cup I_j$, and since $A_k\cap(\lambda_0-\delta,\lambda_0]\not=\emptyset$, it must be that
\[
A_k\cap (A_j\cup A_{j+1}\cup I_j)\not=\emptyset\,,
\]
hence $j=k$ or $j=k-1$.

\medskip\noindent $(vi)$ $\sup I_k\in I_k$: It follows directly from $(iii)$ and $(iv)$.
\end{proof}

\begin{proof}[Proof of Theorem~\ref{Thm:Main}] With the notation of the previous lemma one shows by induction that there exists an increasing sequence $\{\lambda_k\}$,  such that $\lambda_k\in I_k$.

As $\beta^+\in A_0$, by $(ii)$ we can set $\lambda_0=\sup A_0$, and by $(v)$ and $(vi)$, $\lambda_0\in I_0$ and $\lambda_0\le \sup I_0\in I_0$. We use now $(iii)$ and find $\delta>0$ such that
\[
(\sup I_0-\delta\,,\;\sup I_0+\delta)\subset A_0\cup A_1\cup I_0\,.
\]
Since $(\sup I_0,\sup I_0+\delta)\cap A_0=\emptyset$ by the definition of $\lambda_0$ and $(\sup I_0,\sup I_0+\delta)\cap I_0=\emptyset$ by the definition of $\sup I_0$, it must be that
\[
(\sup I_0\,,\;\sup I_0+\delta)\subset A_1
\]
implying that
\[
A_1\not=\emptyset\quad\mbox{and}\quad\lambda_0\le \sup I_0<\lambda_1:=\sup A_1\,.
\]
By $(v)$, $\lambda_1\in I_0\cup I_1$, but as $\sup I_0<\lambda_1$, it must be that $\sup A_1\in I_1$. Then
\[
I_1\quad\mbox{is not empty and}\quad\lambda_1\le \sup I_1\,.
\]
We use again $(iii)$ to find $\delta>0$ such that
\[
(\sup I_1-\delta,\sup I_1+\delta)\subset A_1\cup A_2\cup I_1\,,
\]
and again deduce that
\[
(\sup I_1,\sup I_1+\delta)\subset A_2\,,
\]
hence $A_2\not=\emptyset$ and thanks to $(ii)$ we can set $\lambda_2=\sup A_2$, $\lambda_0<\lambda_1\le \sup I_1<\lambda_2$ and $\lambda_2\in I_2$. We continue this procedure to obtain the infinite strictly increasing sequence $\{\lambda_k\}$, defined by $\lambda_k=\sup A_k$ with $\lambda_k\in I_k$.
\end{proof}
\begin{remark}\label{nono} Note that by Theorem \ref{basic2bis}, $\lambda_k\to +\infty$  if $(H)$ is satisfied. Indeed, if $\lambda_k\to\bar\lambda<\infty$, then, as $f(\bar\lambda)>0$, it must be that $v_{\bar\lambda}$ is oscillatory.
\end{remark}

\section{Examples and concluding remarks}\label{Section:final}

We start this section with the analysis of the examples of Section 2.\medskip

\noindent{\bf Analysis of examples~\ref{ex1},~\ref{ex2} and~\ref{ex3}.} Following some of the ideas in \cite{cmm,pghms}, we consider the following problem, which includes as special cases those three examples:
\[\label{**}
\begin{gathered}
\mbox{div}\big(|x|^k\,|\nabla u|^{p-2}\,\nabla u\big)+|x|^\ell\left(\frac{|x|^s}{1+|x|^s}\right)^{\!\frac\sigma s}\!f(u)=0\quad\mbox{in }\R^d\,,\\
\lim_{|x|\to+\infty} u(x)=0\,,
\end{gathered}
\]
where $d\ge1$, $k$, $\ell\in\R$, and $s$, $\sigma>0$. Here
\[
a(r)=r^{d+k-1}\,,\quad b(r)=r^{d+\ell-1}\left(\frac{r^s}{1+r^s}\right)^{\!\frac\sigma s}\,.
\]

We claim that $(W1)$--$(W4)$ are satisfied if
\beq\label{conds}
\ell>k-p\,,\quad\frac kp+\frac{\ell-1}{p'}\ge1-d\,.
\eeq
First of all, it can be verified that $(W1)$ is a consequence of the first condition in \eqref{conds}. Next we compute
\[
\psi(r) = \left(d - 1 + \frac kp + \frac{\ell}{p'} + \frac{\sigma}{p'}\frac1{1 + r^s}\right)\left(\frac{1+ r^s}{r^s}\right)^{\!\frac\sigma{p\,s}}r^{\frac{k-\ell}p-1}\,.
\]
Since by \eqref{conds} and the assumption that $s$ and $\sigma$ are positive, the three factors appearing in the above expression of $\psi$ are positive and strictly decreasing, thus $(W2)$ holds. In order to check that $(W3)$ holds,
we observe that
\[\label{conds2}
\lim_{r\to0_+}\psi(r)\int_0^r\(\dfrac{b(t)}{a(t)}\)^\frac1p\,dt=\left(d - 1 + \frac kp + \frac{\ell}{p'} + \frac{\sigma}{p'}\right)\lim_{r \to0_+}r^{\frac{k-\ell- \sigma}p - 1}\int_0^r t^{\frac{\sigma- k+\ell}p} \;dt\,,
\]
thus by \eqref{conds} we have
\[\label{conds3}
\lim_{r\to0_+}\psi(r)\int_0^r \Bigl(\dfrac{b}{a}\Bigr)^{1/p}\,dt= \frac{p\,(d-1) + (\ell + \sigma)(p-1)+k}{p + \ell + \sigma - k} >0\,.
\]
A similar calculation gives
\[
\lim_{r\to+\infty}\psi(r)\int_0^r \Bigl(\dfrac{b}{a}\Bigr)^{1/p}\,dt=\Bigl(d-1+\frac kp + \frac{\ell}{p'}\Bigr)\,\frac p{p-k+\ell}>0\,,
\]
hence $(W3)$ is satisfied. Finally, as
\[
a^{p'-1}(r)\,b(r)=\big(a^{\frac1p}(r)\,b^{\frac1{p'}}(r)\big)^{p'}=r^{p'(d-1+\frac kp+\frac{\ell}{p'})}\Bigl(\frac{r^s}{1+r^s}\Bigr)^{\!\frac\sigma s}\,,
\]
by the second condition in \eqref{conds} we conclude that $(W4)$ holds. Note that \eqref{conds} implies that $\ell+d>1/p'$.

Now we analyze the subcriticality assumption $(SC_W)$. By integration by parts we find that
\[
\label{B}B(r)=\int_0^rb(t)\;dt=\frac{r\,b(r)}{d+\ell+\sigma}+\frac{\sigma}{d+\ell+\sigma}\int_0^r\frac{t^s\,b(t)}{1+t^s}\;dt\,.
\]
Set
\[\label{defmur}
\mu(r):=\Bigl(\frac1p-1\Bigr)\Bigl(\frac{B}{b}\Bigr)'(r)+\frac{d+k-1}p\,\frac{B(r)}{r\,b(r)}
\]
where $(d+k-1)/r=a'/a$. As
\[\label{quotb}
\frac{r\,b'(r)}{b(r)}=d+\ell+\sigma-1-\sigma\,\frac{r^s}{1+r^s}\,,
\]
by L'H\^opital's rule we find that
\[\label{quotB}
\Bigl(\frac{B}{b}\Bigr)'(0)=\lim_{r\to0_+}\frac{B(r)}{r\,b(r)}=\lim_{r\to0_+}\frac1{\frac{r\,b'(r)}{b(r)}+1}=\frac1{d+\ell+\sigma}\,,
\]
hence from \eqref{defmur} we obtain that for $d+k>p$,
\[
\limsup_{r\to0_+}\mu(r)=\mu_*=\frac 1p\Bigl(\frac{d+k-p}{d+\ell+\sigma}\Bigr)\,.
\]
For these weights, and $f(s)=|s|^{q-1}s\log|s|$, $1<q+1<1/\mu^*$, the assumptions in $(SC_W)$ are satisfied for any $\alpha\in(0,1)$, and $\mu>\mu^*$ such that $q+1<1/\mu$. Thus we conclude that the equation has nontrivial solutions with any prescribed number of nodes.

\medskip\noindent{\bf Analysis of Example~\ref{Ex:Critical}.} Here we complement the results of Section~\ref{Section:Examples} when $f$ is an odd  function such that $f(s)=|s|^{2^*-2}\,s$ when $|s|\ge\beta$ and $f$ is defined on $(0,\beta)$ so that $(f1)$-$(f2)$ hold with $-\,\beta^-=\beta^+=\beta$. The following result shows that Theorem~\ref{Thm:Main} does not apply if the sub-criticality condition $(SC)$ does not hold.
\begin{prop}\label{Prop:Ex4} There exists a function $f$ as above, and $\lambda_0>\beta$ such that for any $\lambda\ge\lambda_0$ the solution $v_\lambda$ of \eqref{ivp} is everywhere positive.\end{prop}
\begin{proof} For any $\nu\in(-\infty,\lambda)$, let us define $\mbox{\sc r}_\lambda(\nu):=\inf\{r>0\,:\,v_\lambda(r)=\nu\}$. It can be directly verified that
\[
\mbox{\sc r}_\beta^2(\lambda)=\frac{d\,(d-2)}{\beta^\frac2{d-2}\,\lambda^\frac4{d-2}}\(\lambda^\frac2{d-2}-\beta^\frac2{d-2}\)
\quad\mbox{and}\quad\frac12\,|v_\lambda'(\mbox{\sc r}_\beta(\lambda))|^2=\frac{d-2}{2\,d}\,\beta^{2\,\frac{d-1}{d-2}}\(\lambda^\frac2{d-2}-\beta^\frac2{d-2}\)\,.
\]
Using $h$ as in \eqref{h}, \emph{i.e.}~$h(r)=r^{2\,(d-1)}$, we know that $h\,E_\lambda$ is decreasing in $r$, hence
\[
r^{2\,(d-1)}\,E_\lambda(r)\le(\mbox{\sc r}_\beta(\lambda))^{2\,(d-1)}\,E_\lambda(\mbox{\sc r}_\beta(\lambda))\le C(\beta,d)\,\lambda^{-2}
\]
as long as $r\ge\mbox{\sc r}_\beta(\lambda)$ and $|v_\lambda(r)|\le\beta$. Assume now that $v_\lambda$ and $v_{\lambda_0}$ have at least one zero. This means that $\mbox{\sc r}_0(\lambda)$ and $\mbox{\sc r}_0(\lambda_0)$ are finite. From the \emph{separation lemma} \cite[Lemma~4.2]{cghy2-aihp}, if $\lambda>\lambda_0$, it follows that
\[
\mbox{\sc r}_s(\lambda)<\mbox{\sc r}_s(\lambda_0)\quad\mbox{and}\quad\kappa_0:=\mbox{\sc r}_s(\lambda_0)^2\,E_{\lambda_0}(\mbox{\sc r}_s(\lambda_0)<\mbox{\sc r}_s(\lambda)^2\,E_\lambda(\mbox{\sc r}_s(\lambda))
\]
for all $s\in[-\beta,\beta]$. We may assume that $\kappa_0$ is positive (otherwise replace $\lambda_0$ by $\lambda$, and then choose a larger $\lambda$). This means that $\mbox{\sc r}_0(\lambda)<\mbox{\sc r}_0(\lambda_0)$ and $r^2\,E_\lambda(r)>\kappa_0$ for any $r\in[\mbox{\sc r}_\beta(\lambda),\,\mbox{\sc r}_0(\lambda)]$. Hence $E_\lambda(r)>\kappa_0/r^2$ and so we get
\[
r^{2\,(d-2)}\le\frac{C(\beta,d)}{\kappa_0\,\lambda^2}\quad\mbox{and}\quad E_\lambda(r)\ge\kappa_0\,\(\frac{\kappa_0\,\lambda^2}{C(\beta,d)}\)^\frac1{d-2}\quad\mbox{for any }r\in[\mbox{\sc r}_\beta(\lambda),\,\mbox{\sc r}_0(\lambda)]\,.
\]
Now let us assume that we can take $\lambda$ arbitrarily large. In that case we can estimate
\[
\beta=\int_{\mbox{\sc r}_\beta(\lambda)}^{\mbox{\sc r}_0(\lambda)}|v_\lambda'(r)|\;dr\sim\lambda^\frac1{d-2}\({\mbox{\sc r}_0(\lambda)}-{\mbox{\sc r}_\beta(\lambda)}\)\quad\mbox{as }\lambda\to+\infty\,,
\]
thus proving that $\mbox{\sc r}_\beta(\lambda)/\mbox{\sc r}_0(\lambda)=\theta$ for some $\theta\in(0,1)$.

{}From~\eqref{I'}, we know that $E_\lambda'=-\,(d-1)\,|v_\lambda'|^2/r\sim-\,2\,(d-1)\,E_\lambda\,/r$, which shows that
\[
E_\lambda(r)=E_\lambda(\mbox{\sc r}_\beta(\lambda))\(\frac r{\mbox{\sc r}_\beta(\lambda)}\)^{\!-\,2\,(d-1)}(1+o(1)))\quad\mbox{as }\lambda\to+\infty\,.
\]
Using $|v_\lambda'|=\sqrt{2\,E_\lambda}\,(1+o(1)))$, we can compute again
\begin{multline*}
\beta=\int_{\mbox{\sc r}_\beta(\lambda)}^{\mbox{\sc r}_0(\lambda)}|v_\lambda'(r)|\;dr=\sqrt{2\,E_\lambda(\mbox{\sc r}_\beta(\lambda))}\int_{\mbox{\sc r}_\beta(\lambda)}^{\mbox{\sc r}_0(\lambda)}\(\frac r{\mbox{\sc r}_\beta(\lambda)}\)^{\!-\,(d-1)}\;dr\,(1+o(1)))\\
=\sqrt{2\,E_\lambda(\mbox{\sc r}_\beta(\lambda))}\,\mbox{\sc r}_\beta(\lambda)\int_1^{1/\theta}s^{1-d}\;ds\,(1+o(1)))
\,.
\end{multline*}
Since $\lim_{\lambda\to+\infty}\sqrt{2\,E_\lambda(\mbox{\sc r}_\beta(\lambda))}\,\mbox{\sc r}_\beta(\lambda)=(d-2)\,\beta$, we finally get $1=1-\theta^{d-2}$, which is an obvious contradiction.
\end{proof}

\medskip\noindent{\bf Analysis of Example~\ref{Ex:Critical2}.} This case is a limit case for which the critical exponent is achieved, with a nonlinearity which is still slightly sub-critical in some sense, as we shall see below. Although
\[
\gamma+1=\limsup_{|s|\to+\infty}\frac{s\,f(s)}{F(s)}=p^*
\]
and thus usual subcritical assumptions do not hold, $f$ satisfies $(SC)$. As $Q(r)\sim r^d$ for $r$ small, we can rely on Remark~\ref{mufinito-p*} and observe that
\[
\big(F(s)-\frac1{p^*}\,s\,f(s)\big)\,Q\Bigl(\bigl(\tfrac{(1-\alpha)\,s}{\phi_{p'}(f(s))}\bigr)^{1/p'}\Bigr)=\frac\zeta{p^*}\,(\log s)^{\frac{d-p}p\,\zeta-1}\to+\infty\quad\mbox{as }s\to+\infty\,.
\]
Problem \eqref{eq2} has therefore bounded states with any prescribed number of nodes.

\subsection*{Concluding remarks.} We conclude this section with some remarks concerning the support of a solution to \eqref{eq2} and  the case $\mu^*=0$.

\subsubsection*{Compactly supported solutions and double zeros} It should be noticed that the proof given in \cite[Proposition 2]{fls} can be adapted here to prove that a necessary and sufficient condition for a solution of \eqref{eq2} to be compactly supported is that $s\mapsto|F(s)|^{-1/p}$ is integrable in a neighborhood of $s=0$. In that case, the value of $r$ corresponding to the boundary of the support, if it is bounded,  is a double zero of the solution. For completeness, let us give some hints on the proof.

If $v_\infty$ solves \eqref{asym12}, then we can prove that $|v_\infty'|^p/p'+F(v_\infty)=0$ by multiplying the equations by $v_\infty'$ and taking a primitive. This allows to prove that
\[
r-r_0=\int_0^{v_\infty(r)}|F(s)|^{-1/p}\,ds
\]
if $r$ is in a left neighborhood of $r_0$ on which $v_\infty(r)>0$ and such that $v_\infty(r_0)=0$. By monotonicity, the above relation can be inverted to provide an expression of $v_\infty(r)$.

If we consider a solution to \eqref{ivp}, it can be proved that it is compactly supported if $s\mapsto|F(s)|^{-1/p}$ is integrable using comparison methods. For detailed results, we primarily refer to \cite{Pucci-Serrin-Zou99}, to \cite{Benilan-Brezis-Crandall75,Cortazar-Elgueta-Felmer96,Balabane-Dolbeault-Ounaies01} in case $p=2$ and to \cite{Vazquez84,Serrin-Zou99,Pucci-Serrin-Zou99,Pucci-Serrin00,Felmer-Quaas01,GarciaHuidobro-Manasevich-Serrin-Tang-Yarur01} in the general case.

\medskip If $s\mapsto|F(s)|^{-1/p}$ is not integrable in a neighborhood of $s=0$, solutions cannot be compactly supported and double zeros appear at $+\infty$, that is,
\[
\lambda\in I_k\quad\Longrightarrow\quad r_\lambda(0)=+\infty\,.
\]
This is of course consistent with Hopf's lemma, as established for instance in \cite{Vazquez84}. Assumption $(H)$ provides a sufficient condition for the solutions to have only a finite number of nodes. We do not claim that it is optimal.

For completeness, it has to be noted that the behavior of $f$ may differ in right and left neighborhoods of $s=0$. The discussion of the various cases is left to the reader.

\subsubsection*{Location of the last node} We have observed that $\lim_{\lambda\to+\infty}r(\lambda)=+\infty$, thus showing that the support of compactly supported solutions becomes larger and larger as $\lambda\to+\infty$. Such solutions have been studied in \cite{Balabane-Dolbeault-Ounaies01,dghm}.

For solutions which are not compactly supported, this means that the largest zero goes to~$\infty$ as $\lambda\to+\infty$. Notice that this does not give a lower bound on the number of zeros in a fixed interval containing $0$, which is consistent with the results of \cite{GMZ}. This means that even in the sub-critical regime the distance that separates two consecutive zeros or the lowest zero from the origin, $r=0$, may become arbitrarily small as $\lambda\to+\infty$.

Still, from Lemmata \ref{AngularVelocity} and~\ref{nw1}, it can be deduced that the last nodes become arbitrarily large when $\lambda\to+\infty$

\subsubsection*{The case $\mu^*=0$}
In this case, we claim that $(SC)$ is always satisfied for $f$ such that $f(s)\sim C\,s^\gamma$ as $|s|\to+\infty$. Indeed, we can choose any $\alpha\in(0,1)$, and $\mu\in(0,\min\{1/(\gamma+1),1/p\})$, small. From the definition of $\mu^*$, it follows that there exists $C_0>0$ such that $Q(r)\ge C_0\,r^{p+\varepsilon}$ for $r$ small, with $\varepsilon=\mu\,p^2/(1-\mu\,p)$. Hence
\ben
\Bigl(F(s_2)-\mu\,s_2\,f(s_2)\Bigr)\,Q\!\left(\Bigl(\frac{(1-\alpha)\,s}{\phi_{p'}(f(s_1))}\Bigr)^{\!1/p'} \right)\ge
C\,\big(\tfrac1{\gamma+1}-\mu\big)\,s^{\gamma+1}\,Q\!\left((1-\alpha)\,s^{[1-(p'-1)\,\gamma]\,\frac1{p'}}\right)
\een
for some positive constant $C$ and for any $s_1,s_2\in[\alpha\,s,s]$.
If $1-(p'-1)\gamma\ge 0$, then \eqref{c0} clearly holds as $Q$ is an increasing function. If $1-(p'-1)\,\gamma< 0$, then for $s$ large enough,
\[
s^{\gamma+1}\,Q\((1-\alpha)\,s^{[1-(p'-1)\,\gamma]\,\frac1{p'}}\)\ge C\,s^{\gamma+1+(p+\varepsilon)\,[1-(p'-1)\,\gamma]\frac1{p'}}\,,
\]
and since
\[
\gamma+1+(p+\varepsilon)\,[1-(p'-1)\,\gamma]\,\frac1{p'}=p+\frac\varepsilon{p'}\,[1-(p'-1)\,\gamma]>0
\]
for $\mu$ small enough, that is, for $\varepsilon$ small, our claim follows.

\appendix\section{Existence and uniqueness results}\label{Section:Uniqueness}

This Appendix is devoted to the proof of Proposition~\ref{Prop:Existence-Uniqueness}. By Proposition~\ref{basic1}, if a solution~$v=v_\lambda$ to~\eqref{ivp} exists on some interval $[0,r_0)$, then $|v_\lambda'|+|v_\lambda|\le C_\lambda$ in $[0,r_0)$ for some positive constant~$C_\lambda$. Hence, if $v_\lambda$ can be defined in an interval of the form $[0,r_0]$ for $r_0>0$, then this solution can be extended to $[0,\infty)$. Such an existence result (and a uniqueness result conditional to the values of $v_\lambda$) is proved in \cite[Proposition 9.1, 9.2]{pghms} so we omit it.

To study the unique extendibility situation at points $r_0>0$ we re-write \eqref{ivp} in the form
\be\label{cg1}
\begin{cases}v'=\phi_{p'}(w)\\
w'=\displaystyle-\frac{q'}q\,w-f(v)\\
v(r_0)=v_0\,,\quad w(r_0)=w_0
\end{cases}
\ee
and observe that the only delicate situations occur at points $(v_0,w_0)$ with either $v_0=0$ or $w_0=0$. Our approach is based on ideas that can be traced back to \cite{MR682268,MR829369} (also see~\cite{pghms}).

When $p\le 2$, a solution of \eqref{cg1} can be uniquely extended until it reaches a double zero, as in this case the right hand side is Lipschitz in a neighborhood of any point $(v_0,w_0)$ with $v_0\not=0$. When $p>2$, the argument does not apply because $\phi_{p'}$ is not locally Lipschitz near~$0$. Nevertheless unique extendibility still holds if $f(v_0)\not=0$ and $w_0=0$. Let us prove it.

Assume for simplicity that $f(v_0)<0$. From the second equation in \eqref{cg1} and using the fact that $w(r_0)=w_0=0$, we have that
\[
w'(r)=-\frac{q'}q\,w-f(v)\ge \frac12\,|f(v_0)|\quad\mbox{if}\quad |r-r_0|<\delta
\]
if $\delta>0$ is small enough, then implying that
\[\label{expl}
|w(r)|\ge|f(v_0)|\,\frac{|r-r_0|}2\quad\mbox{for all $r\in(r_0-\delta, r_0+\delta)$}\,.
\]
Hence, if $(v_1,w_1)$ and $(v_2,w_2)$ are two solutions of \eqref{cg1}, then for $r>r_0$, from the mean value theorem, using the fact that $p'-2<0$, we have
\ben
|(v_1'-v_2')(r)|=|(\phi_{p'}(w_1)-\phi_{p'}(w_2))(r)|&=&(p'-1)\,|\xi|^{p'-2}\,|(w_1-w_2)(r)|\\
&\le& C\,(r-r_0)^{p'-2}\,|(w_1-w_2)(r)|
\een
for some positive constant $C$ and thus
\[
|(v_1-v_2)(r)|\le C\,(r-r_0)^{p'-1}\,\|w_1-w_2\|
\]
where $\|\cdot\|$ represents the usual sup norm in $C([r_0-\delta,r_0+\delta];\R)$. Also from the second equation in \eqref{cg1}, using that $f$ is locally Lipschitz we find that
\[
|(w_1-w_2)(r)|\le\int_{r_0}^r|(w_1'-w_2')(s)|\;ds\le C\,\|w_1-w_2\|\,(r-r_0)+K\,\|v_1-v_2\|\,(r-r_0)
\]
for some positive constant $K$. Adding up these two last inequalities we have that
\[
\|v_1-v_2\|+\|w_1-w_2\|\le C\,(\delta^{p'-1}+\delta)\,\|w_1-w_2\|+K\,\delta\,\|v_1-v_2\|\,,
\]
or equivalently,
\[
(1-K\,\delta)\,\|v_1-v_2\|+\big(1-\,C\,(\delta^{p-1}+\delta)\big)\,\|w_1-w_2\|\le 0\,.
\]
Hence, choosing $\delta$ small enough we deduce $v_1=v_2$ and $w_1=w_2$ and unique extendibility follows.

Finally, if a solution reaches the value zero with a nonzero slope, (that is, $v_0=0$, $w_0\not=0$), then this solution can be uniquely continued by considering its inverse, \emph{i.e.}~the function $r=r(s)$ such that $v(r(s))=s$, which satisfies the equation
\be\label{r-eq0}
\begin{cases}\displaystyle(p-1)\,r''=\frac{q'}q\,r'(s)+f(s)\,|r'(s)|^p\,,\\
r(0)=r_0>0\,,\quad r'(0)=t_0\not=0\,,
\end{cases}
\ee
or equivalently,
\[
(p-1)\,r''(s)=\frac d{ds}\log\big(q(r(s))\big)+f(s)\,|r'(s)|^p\,.
\]
We recall now that $r(0)=r_0>0$, $r'(0)\not=0$, so everything that follows is well defined in a neighborhood of $s=0$. Indeed since $r$ and $r'$ are continuous, there exists $\delta_0>0$ such that for $|s|<\delta_0$ it holds that $r_0/2<r(s)\le3\,r_0/2$ and $|t_0|/2<|r'(s)|<3\,|t_0|/2$. Let $s\in(-\delta_0,\delta_0)$. Then integrating over $(0,s)$ we get
\[
(p-1)\,\big(r'(s)-t_0\big)=\log\(\frac{q(r(s)}{q(r_0)}\)+\int_0^sf(\xi)\,|r'(\xi)|^p\,d\xi\,.
\]
Let now $r_1$, $r_2$ be two solutions of \eqref{r-eq0} defined in $(-\delta_0,\delta_0)$ where $\delta_0=\min\{\delta_{1,0},\,\delta_{2,0}\}$ with obvious notations. Then, using that
\[
\big|\ |r_1'(\xi)|^p-|r_2'(\xi)|^p\big|=p\,\phi_p(|\eta|)\,|r_1'(\xi)-r_2'(\xi)|\quad\mbox{for some $|\eta|\in\big(|t_0|/2,\,3\,|t_0|/2\big)$}\,,
\]
we obtain
\[
|r_1'(s)-r_2'(s)|\le A_0\,|r_1(s)-r_2(s)|+B_0\,\int_0^s|f(\xi)|\,|r_1'(\xi)-r_2'(\xi)|\;d\xi
\]
for some positive constants $A_0$, $B_0$ depending on $r_0$, $t_0$ and $\delta_0$, and thus, if $|s_0|\le \delta_0$,
\[
\sup_{s\in[-s_0,s_0]}|r_1'(s)-r_2'(s)|\le A_0\sup_{\xi\in[-s_0,s_0]}|r_1'(\xi)-r_2'(\xi)|\,|s_0|+B_0\sup_{\xi\in[-s_0,s_0]}|r_1'(\xi)-r_2'(\xi)|\,|F(s_0)|
\]
that is,
\[
\sup_{s\in[-s_0,s_0]}|r_1'(s)-r_2'(s)|\,\Bigl(1- A_0\,|s_0|-B_0\,|F(s_0)|\Bigr)\le0\,.
\]
Since $F(0)=0$, by choosing $|s_0|$ small enough, we must have $r_1'(s)\equiv r_2'(s)$ in $[-s_0,s_0]$ and hence also $r_1(s)\equiv r_2(s)$ in $[-s_0,s_0]$.
\unskip\null\hfill$\square$\vskip 0.3cm


\end{document}